\documentclass[a4paper, 11pt]{amsart}   
\usepackage{mathptmx, amssymb,amscd,latexsym, eulervm}   
\usepackage{amsmath}
\usepackage{amsthm}
\usepackage{mathdots}
\usepackage[pagebackref,colorlinks=true,linkcolor=blue,urlcolor=blue]{hyperref}
\usepackage{color}
\usepackage[onehalfspacing]{setspace}
\usepackage{tabularx}
\usepackage{amsfonts}
\usepackage{paralist}
\usepackage{aliascnt}
\usepackage[initials, lite]{amsrefs}
\usepackage{amscd}
\usepackage{blkarray}
\usepackage{mathbbol}
\usepackage{setspace}
\usepackage[inner=2.5cm,outer=2.5cm, bottom=3.2cm]{geometry}
\usepackage{tikz, tikz-cd}
\usepackage{calligra,mathrsfs}

\usepackage{tikz}
\usetikzlibrary{matrix}
\usetikzlibrary{arrows,calc}
\allowdisplaybreaks

\BibSpec{collection.article}{%
	+{}  {\PrintAuthors}                {author}
	+{,} { \textit}                     {title}
	+{.} { }                            {part}
	+{:} { \textit}                     {subtitle}
	+{,} { \PrintContributions}         {contribution}
	+{,} { \PrintConference}            {conference}
	+{}  {\PrintBook}                   {book}
	+{,} { }                            {booktitle}
	+{,} { }                            {series}
	+{, vol.} { }                            {volume}
	+{,} { }                            {publisher}
	+{,} { \PrintDateB}                 {date}
	+{,} { pp.~}                        {pages}
	+{,} { }                            {status}
	+{,} { \PrintDOI}                   {doi}
	+{,} { available at \eprint}        {eprint}
	+{}  { \parenthesize}               {language}
	+{}  { \PrintTranslation}           {translation}
	+{;} { \PrintReprint}               {reprint}
	+{.} { }                            {note}
	+{.} {}                             {transition}
	+{}  {\SentenceSpace \PrintReviews} {review}
}

\newcommand{\kk}{\mathbb{k}}

\newcommand{\NN}{\normalfont\mathbb{N}}
\newcommand{\ZZ}{{\normalfont\mathbb{Z}}}

\newcommand{\PP}{\normalfont\mathbb{P}}

\newcommand{\mm}{{\normalfont\mathfrak{m}}}

\newcommand{\QQ}{\mathbb{Q}}
\newcommand{\pp}{\mathfrak{p}}
\newcommand{\aaa}{\mathfrak{a}}
\newcommand{\bb}{\mathfrak{b}}

\newcommand{\nn}{{\mathfrak{n}}}

\newcommand{\reg}{\normalfont\text{reg}}

\newcommand{\depth}{\normalfont\text{depth}}

\newcommand{\Ext}{\normalfont\text{Ext}}

\newcommand{\Ker}{\normalfont\text{Ker}}

\newcommand{\Fib}{{\normalfont\text{Fib}}}

\newcommand{\HS}{{\normalfont\text{HS}}}
\newcommand{\IM}{\normalfont\text{Im}}

\newcommand{\II}{\mathcal{I}}

\newcommand{\ob}{{\normalfont\text{ob}}}
\newcommand{\Hom}{{\normalfont\text{Hom}}}

\newcommand{\bbA}{\mathbb{A}}

\newcommand{\OO}{\mathcal{O}}

\newcommand{\FF}{\mathcal{F}}
\newcommand{\HL}{\normalfont\text{H}_{\mm}}
\newcommand{\HH}{\normalfont\text{H}}

\newcommand{\iniTerm}{\normalfont\text{in}}

\newcommand{\Proj}{\normalfont\text{Proj}}
\newcommand{\Spec}{{\normalfont\text{Spec}}}

\newcommand{\FLoc}{{\normalfont\text{FLoc}}}

\newcommand{\Art}{{\normalfont\text{Art}}}

\newcommand{\Hilb}{\normalfont\text{Hilb}}

\newtheorem{theorem}{Theorem}[section]

\newtheorem{headthm}{Theorem}

\newaliascnt{headcor}{headthm}

\aliascntresetthe{headcor}

\newaliascnt{headconj}{headthm}

\aliascntresetthe{headconj}

\newaliascnt{corollary}{theorem}
\newtheorem{corollary}[corollary]{Corollary}
\aliascntresetthe{corollary}

\newaliascnt{claim}{theorem}

\aliascntresetthe{claim}

\newaliascnt{lemma}{theorem}
\newtheorem{lemma}[lemma]{Lemma}
\aliascntresetthe{lemma}

\newaliascnt{conjecture}{theorem}

\aliascntresetthe{conjecture}

\newaliascnt{proposition}{theorem}
\newtheorem{proposition}[proposition]{Proposition}
\aliascntresetthe{proposition}

\theoremstyle{definition}
\newaliascnt{definition}{theorem}
\newtheorem{definition}[definition]{Definition}
\aliascntresetthe{definition}

\newaliascnt{notation}{theorem}

\aliascntresetthe{notation}

\newaliascnt{example}{theorem}
\newtheorem{example}[example]{Example}
\aliascntresetthe{example}

\newaliascnt{examples}{theorem}

\aliascntresetthe{examples}

\newaliascnt{remark}{theorem}
\newtheorem{remark}[remark]{Remark}
\aliascntresetthe{remark}

\newaliascnt{question}{theorem}
\newtheorem{question}[question]{Question}
\aliascntresetthe{question}

\newaliascnt{questions}{theorem}

\aliascntresetthe{questions}

\newaliascnt{problem}{theorem}

\aliascntresetthe{problem}

\newaliascnt{construction}{theorem}

\aliascntresetthe{construction}

\newaliascnt{setup}{theorem}
\newtheorem{setup}[setup]{Setup}
\aliascntresetthe{setup}

\newaliascnt{algorithm}{theorem}

\aliascntresetthe{algorithm}

\newaliascnt{observation}{theorem}

\aliascntresetthe{observation}

\newaliascnt{defprop}{theorem}

\aliascntresetthe{defprop}

\DeclareFontFamily{OT1}{pzc}{}
\DeclareFontShape{OT1}{pzc}{m}{it}{<-> s * [1.100] pzcmi7t}{}
\DeclareMathAlphabet{\mathchanc}{OT1}{pzc}{m}{it}

\DeclareMathOperator{\fFib}{\mathchanc{Fib}}

\DeclareMathOperator{\fHS}{\mathchanc{HS}}

\def\equationautorefname~#1\null{(#1)\null}
\def\sectionautorefname~#1\null{Section #1\null}
\def\subsectionautorefname~#1\null{\S #1\null}

\begin{document}

	\title{A local study of the fiber-full scheme}

\author{Yairon Cid-Ruiz}
\address{Department of Mathematics, KU Leuven, Celestijnenlaan 200B, 3001 Leuven, Belgium}
\email{yairon.cidruiz@kuleuven.be}
	
	\author{Ritvik Ramkumar}
	\address{Department of Mathematics, University of California at Berkeley, Berkeley, CA, 94720, USA}
	\email{ritvik@math.berkeley.edu}


	\date{\today}

	\begin{abstract}
		We study some of the local properties of the fiber-full scheme, which is a fine moduli space that generalizes the Hilbert scheme by  parametrizing closed subschemes with  prescribed cohomological data.
		As a consequence, we provide sufficient conditions for cohomology to remain constant under Gr\"obner degenerations. 
		We also describe a tangent-obstruction theory for the fiber-full scheme in analogy with the one for the Hilbert scheme.
	\end{abstract}

	\maketitle
	
\section{Introduction}

In previous recent work \cite{FIBER_FULL_SCHEME}, we introduced the \emph{fiber-full scheme} which can be seen as the parameter space that generalizes the Hilbert scheme by controlling the entire cohomological data.
Here we continue the study of this scheme from a local point of view. 
Before discussing our findings regarding the fiber-full scheme, we choose to address the following fundamental questions that are among the main motivations of our work.

Let $\kk$ be a field, $R$ be a standard graded polynomial ring $R = \kk[x_0,\ldots,x_r]$, $\mm = (x_0,\ldots,x_r) \subset R$ be the graded irrelevant ideal and $S = R[t] \cong R \otimes_{\kk} \kk[t]$.

\begin{question}
		Let $\mathcal{I} \subset S$ be a one-parameter flat family of homogeneous ideals with special fiber $I_0 \subset R$ and general fiber $I_1 \subset R$. 
		Under which conditions imposed \emph{only} on the special fiber $I_0 \subset R$ do we have that 
		$
		\dim_\kk\left(\left[\HL^i(R/I_0)\right]_\nu\right) = \dim_\kk\left(\left[\HL^i(R/I_1)\right]_\nu\right)
		$
		for all $i \ge 0$ and $\nu \in \ZZ$?
\end{question}

We say that a one-parameter flat family $\II \subset S$ \emph{has constant cohomology} if it satisfies the conclusion of the above question.
Recall that if $\II \subset S$ has constant cohomology then important invariants are equal for the special fiber $I_0 \subset R$ and the general fiber $I_1 \subset R$ (for instance, we obtain equalities in Castelnuovo-Mumford regularity and in depth, i.e., $\reg(R/I_0) = \reg(R/I_1)$ and $\depth(R/I_0) = \depth(R/I_1)$).
By upper semicontinuity, we always have that $\dim_\kk\left(\left[\HL^i(R/I_0)\right]_\nu\right) \ge \dim_\kk\left(\left[\HL^i(R/I_1)\right]_\nu\right)$.
A particular case of major interest in the following sub-question about \emph{Gr\"obner degenerations}.

\begin{question}
	\label{question_2}
	Let $I \subset R$ be a homogeneous ideal and $\iniTerm_>(I) \subset R$ be the initial ideal with respect to a monomial order $>$ on $R$.
	Under which conditions imposed \emph{only} on the initial ideal $\iniTerm_>(I) \subset R$ do we have that 
	$
	\dim_\kk\left(\left[\HL^i(R/I)\right]_\nu\right) = \dim_\kk\left(\left[\HL^i(R/\iniTerm_>(I))\right]_\nu\right)
	$
	for all $i \ge 0$ and $\nu \in \ZZ$?
\end{question}

Important progress was made recently by Conca and Varbaro in \cite{CONCA_VARBARO} as they showed that the conclusion of \autoref{question_2} is satisfied when $\iniTerm_>(I)$ is squarefree.
Our main contribution towards the above questions is the following theorem.

\begin{headthm}[{\autoref{thm_ one_param_ideals}}]
	\label{thmA}
	Under the above notation, let $\mathcal{I} \subset S$ be a one-parameter flat family of homogeneous ideals with special fiber $I_0 \subset R$ and general fiber $I_1 \subset R$. 
	If the zero map is the only zero-degree homogeneous $R$-linear map from $\HL^{i-1}(R/I_0)$ to $\HL^{i}(R/I_0)$ for all $1 \le i \le r$, i.e.,
	$$
	\left[\Hom_R\Big(\HL^{i-1}(R/I_0),\, \HL^{i}(R/I_0)\Big)\right]_0 = 0 \quad \text{for all $1 \le i \le r$},
	$$
	then the family $\mathcal{I}$ has constant cohomology.
\end{headthm}

This result provides an abundance of new families of ideals such that flat degenerations to them preserve cohomology. Notice that the condition of \autoref{thmA} is trivially satisfied when the special fiber $I_0 \subset R$ is Cohen-Macaulay. As pointed out by the following examples, there are many non-squarefree monomial ideals that satisfy the criterion of  \autoref{thmA}.

\begin{example}[{\autoref{examp_first}, \autoref{examp_second}}]
	Consider the (non-squarefree) monomial ideals 
	$$
	I_1 = (x^2,y^2,xy,xz,yz), \quad I_2 = (x^2y, xz^2, y^2z, z^3) \quad  \text{and} \quad I_3 = (xy^2, xz)
	$$
	in $R = \kk[x,y,z]$, which are not Cohen-Macaulay.
	These three ideals satisfy the criterion of \autoref{thmA}.
	Thus, for any homogeneous ideal $I \subset R$ such that  $\iniTerm_>(I)$ equals $I_1$, $I_2$ or $I_3$, we have that  $
	\dim_\kk\left(\left[\HL^i(R/I)\right]_\nu\right) = \dim_\kk\left(\left[\HL^i(R/\iniTerm_>(I))\right]_\nu\right)
	$
	for all $i \ge 0$ and $\nu \in \ZZ$.
\end{example}

The result of \autoref{thmA} does not cover all squarefree monomial ideals; 
in \autoref{examp_sqrfree_bad} we give a squarefree monomial ideal that does not satisfy the condition of \autoref{thmA}. 

From another point of view, a typical example of a non-Cohen Macaulay ideal of dimension at least $2$, is one where the local cohomology is only non-vanishing at the the depth and dimension, both of which are different. Even more significantly, any ideal whose local cohomology is vanishing in alternate degrees also satisfies the hypothesis of \autoref{thmA}. In both of these situations, the ideals satisfy the criterion of \autoref{thmA}. For instance, generic residual intersections \autoref{examp_residual} and many semigroup rings \cite{TRUNG_HOA}. 

\medskip

Our proof of \autoref{thmA} follows directly from the local study of the fiber-full scheme that we make in this paper.
This is hardly a surprise because, by definition, the fiber-full scheme is the parameter space controlling the entire cohomological data.
The fiber-full scheme $\Fib_{\PP_A^r/A}^\mathbf{h}$ is a fine moduli space parametrizing all closed subschemes $Z \subset \PP_A^r$ such that $\HH^i(\PP_{A}^r, \OO_Z(\nu))$ is a locally free $A$-module of rank equal to $h_i(\nu)$, where $\mathbf{h} = (h_0,\ldots,h_r) : \ZZ^{r+1} \rightarrow \NN^{r+1}$ is a fixed tuple of functions (see \cite{FIBER_FULL_SCHEME}).

Since we are interested in arbitrary ideals and not just the saturated ones, we introduce a version of the fiber-full scheme that works for homogeneous ideals. 
Accordingly, we obtain a cohomological stratification of the version of the Hilbert scheme introduced by Haiman and Sturmfels \cite{HAIMAN_STURMFELS}.
Let $A$ be a Noetherian ring, $T = A[x_0,\ldots,x_r]$ be a standard graded polynomial ring and $\mm = (x_0,\ldots,x_r) \subset T$.
For a given function $g : \NN \rightarrow \NN$, the Hilbert scheme $\HS_{T/A}^g$ of Haiman and Sturmfels parametrizes all homogeneous ideals $I \subset T$ such that $\left[T/I\right]_\nu$ is a locally free $A$-module of rank equal to $g(\nu)$ for all $\nu \in \NN$.
We stratify the scheme $\HS_{T/A}^g$ as follows.
Given a function $g : \NN \rightarrow \NN$ and a tuple of functions $\mathbf{h} = (h_0,\ldots,h_{r+1}) : \ZZ^{r+2} \rightarrow \NN^{r+2}$, we develop the fiber-full scheme $\Fib_{T/A}^{g,\mathbf{h}}$  parametrizing all homogeneous ideals $I \subset T$ that satisfy the two conditions:
\begin{enumerate}[(a)]
	\item $
	\text{$\left[T/I\right]_\nu$ is a locally free $A$-module of rank equal to $g(\nu)$ for all $\nu \in \NN$}
	$, and
	\item $\text{$\left[\HL^i(T/I)\right]_\nu$ is a locally free $A$-module of rank equal to $h_i(\nu)$ for all $i \ge 0$ and $\nu \in \ZZ$}$.
\end{enumerate}
In \autoref{thm_homog_fib_sch}, by using the techniques developed in \cite{FIBER_FULL_SCHEME}, we show the existence of the parameter space $\Fib_{T/A}^{g,\mathbf{h}}$.
Note that 
$$
\Fib_{T/A} \,=\, \bigsqcup_{\substack{g : \NN \rightarrow \NN\\ \mathbf{h} : \ZZ^{r+2} \rightarrow \NN^{r+2}}} \Fib_{T/A}^{g,\mathbf{h}}
$$
is the fiber-full scheme that parametrizes all homogeneous ideals giving a flat quotient with flat local cohomology modules. 

Our first main result towards a local study of the fiber-full scheme is determining its tangent space.

\begin{headthm}[\autoref{thm_tan_space_loc}]
	Under the above notation, let $I_0 \subset R$ be a homogeneous ideal. 
	The tangent space of $\Fib_{R/\kk}$ at the corresponding point $[I_0]$ is given by 
	$$
	T_{[I_0]} \Fib_{R/\kk} = \big\{\varphi \in \big[\Hom(I_0,R/I_0)\big]_0 \,\mid\, \HH^i_{\mm}(\varphi) = 0 \text{ for all $i \ge 0$}\big\},
	$$
	where $\HL^i(\varphi)$ denotes the natural map $\HL^i(\varphi) : \HL^i(I_0) \rightarrow \HL^i(R/I_0)$  induced in cohomology.
\end{headthm}

Since $T_{[I_0]} \HS_{R/\kk} = \left[\Hom_R(I_0, R/I_0)\right]_0$, one could rephrase the above theorem by saying that the tangent space of $\Fib_{R/\kk}$ is given by the tangent vectors of $\HS_{R/\kk}$ that are trivial from a cohomological point of view.
This seems quite intuitive since by definition cohomologies remain constant over the entire fiber-full scheme $\Fib_{R/\kk}^{g, \mathbf{h}}$.
In \autoref{thm_tan_space}, we give a similar computation for the tangent space of the fiber-full scheme $\Fib_{\PP_{\kk}^r/\kk}$.

\medskip
We also provide a description of a tangent-obstruction theory for the fiber-full scheme in analogy with the one for the Hilbert scheme
(see, e.g., \cite[Theorem 6.2]{hartshorne2010deformation}).
Let $(C', \nn',\kk)$ and $(C, \nn,\kk)$ be Artinian local rings with residue field $\kk$, and suppose that 
$
0 \rightarrow \aaa \rightarrow C' \rightarrow C \rightarrow 0
$
is a short exact sequence with $\nn' \aaa = 0$.
Let $T = C[x_0,\ldots,x_r]$ and $T' = C'[x_0,\ldots,x_r]$ be  standard graded polynomial rings.
Let $I \subset T$ be a homogeneous ideal such that $[I] \in \Fib_{T/C}$, and denote by $I_0 \subset R$  the homogeneous ideal $I_0 = I \otimes_{C} \kk$ (i.e., $I \subset T$ is a fiber-full extension of $I_0 \subset R$ over $C$).
We seek all homogeneous ideals $I' \subset T'$ such that $[I'] \in \Fib_{T'/C'}$ and $I \cong I' \otimes_{C'} C$.
The following result describes  the fiber-full extensions of $I$ over $C'$.

\begin{headthm}[\autoref{thm_small_extensions_loc}]
	Under the notations above, the set of fiber-full extensions of $[I] \in \Fib_{T/C}$ over $C'$
	$$
	\left\lbrace [I'] \in \Fib_{T'/C'} \;\text{ such that }\;  I' \otimes_{C'} C \cong I \right\rbrace
	$$
	is a pseudotorsor under the action of the subgroup of $\big[\Hom_R\left( I_0, R/I_0 \otimes_\kk \aaa \right)\big]_0$ given by
	$$
	\big\lbrace \varphi \in \big[\Hom_R\left( I_0, R/I_0 \otimes_\kk \aaa \right)\big]_0 \;\mid\; \HL^i(\varphi)=0 \text{ for all $i \ge 0$} \big\rbrace, 
	$$
	where $\HL^i(\varphi)$ denotes the natural map $\HL^i(\varphi) : \HL^i(I_0) \rightarrow \HL^i(R/I_0 \otimes_\kk \aaa)$ induced in cohomology.
\end{headthm}

In \autoref{thm_obs_loc}, we determine sufficient conditions for the existence of obstructions in the fiber-full scheme $\Fib_{R/\kk}$.
Furthermore, in \autoref{subsect_fib_full_sch} we provide a similar study for a tangent-obstruction theory for the fiber-full scheme $\Fib_{\PP_{\kk}^r/\kk}$ parametrizing closed subschemes.

\medskip

The organization of this paper is as follows.
In \autoref{sect_prelim}, we set up the notation used throughout the document and we recall some required results.
In \autoref{sect_hom_fib_full}, we show the existence of the fiber-full scheme parametrizing homogeneous ideals.
We study tangent-obstruction theories for the fiber-full schemes $\Fib_{R/\kk}$ and $\Fib_{\PP_{\kk}^r/\kk}$ in \autoref{sect_tan_obs}.
Our results regarding  one-parameter flat families are given in \autoref{sect_one_param}.
In \autoref{sect_comparison_local_functors}, we show that the local deformation functors of $\Fib_{R/\kk}$ and $\HS_{R/\kk}$ coincide under the condition of \autoref{thmA}.
Finally, in \autoref{sect_examp} we provide several examples illustrating our work.

\medskip
\noindent
\textbf{Acknowledgments.} 
Y.C.R. was partially supported by an FWO Postdoctoral Fellowship (1220122N).

\section{Preliminaries}
\label{sect_prelim}

In this short section, we fix the notation and recall some necessary results that were obtained in the previous papers \cite{FIBER_FULL, FIBER_FULL_SCHEME}.
Here we review some basic results regarding \emph{fiber-full modules and sheaves}, and the corresponding parameter space that we introduced: \emph{the fiber-full scheme}.
Throughout this section the following setup is fixed. 

\begin{setup}
	\label{setup_prelim}
	Let $A$ be a Noetherian ring, $R = A[x_0,\ldots,x_r]$ be a standard graded polynomial ring, and $\mm = (x_0,\ldots, x_r) \subset R$ be the graded irrelevant ideal.
\end{setup}

The main interest in fiber-full modules and sheaves comes from the fact that all their cohomologies are flat over the base ring. 

\begin{definition}
 	A finitely generated graded $R$-module is said to be \emph{fiber-full over $A$} if $M$ is flat over $A$ and $\HL^i(M)$ is flat over $A$ for all $i \ge 0$.
\end{definition}

\begin{definition}
	A coherent sheaf $\FF$ on $\PP_A^r = \Proj(R)$ is said to be \emph{fiber-full over $A$} if $\HH^i(\PP_A^r, \FF(\nu))$ is flat over $A$ for all $i \ge 0$ and $\nu \in \ZZ$.
\end{definition}

The following result is a fundamental criterion to detect fiber-full modules and sheaves. It is also the primary reason for choosing the term \emph{fiber-full} (see also \cite[Chapter 3]{CONFERENCE_LEVICO}).

\begin{theorem}[{\cite[Theorem A]{FIBER_FULL}}]
	\label{thm_fib_full_mod}
	Let $M$ be a finitely generated graded $R$-module. 
	Then, the following results are equivalent: 
	\begin{enumerate}[\rm (i)]
		\item $M$ is fiber-full over $A$.
		\item $M$ is flat over $A$ and $\HL^i\left(M \otimes_{A} A_\pp/\pp^qA_\pp\right)$ is $(A_\pp/\pp^qA_\pp)$-flat for all $i \ge 0$,  $q \ge 1$ and $\pp \in \Spec(A)$.
		\item $M$ is flat over $A$ and the natural map $\HL^i\left(M \otimes_{A} A_\pp/\pp^qA_\pp\right) \rightarrow \HL^i\left(M \otimes_{A} A_\pp/\pp A_\pp\right)$ is surjective for all $i \ge 0$, $q \ge 1$ and $\pp \in \Spec(A)$.
	\end{enumerate}
\end{theorem}

\begin{theorem}[{\cite[Theorem 4.2]{FIBER_FULL_SCHEME}}]
	\label{thm_fib_full_shv}
	Let $\FF$ be a coherent sheaf on $\PP_A^r$. 
	Then, the following results are equivalent: 
	\begin{enumerate}[\rm (i)]
		\item $\FF$ is fiber-full over $A$.
		\item $\FF$ is flat over $A$ and $\HH^i\left(\PP_A^r, \FF(\nu) \otimes_{A} A_\pp/\pp^qA_\pp\right)$ is $A_\pp/\pp^qA_\pp$-flat for all $i \ge 0$, $q \ge 1$, $\nu \in \ZZ$ and $\pp \in \Spec(A)$.
		\item $\FF$ is flat over $A$ and the natural map $\HH^i\left(\PP_A^r, \FF(\nu) \otimes_{A} A_\pp/\pp^qA_\pp\right) \rightarrow \HH^i\left(\PP_A^r, \FF(\nu) \otimes_{A} A_\pp/\pp A_\pp\right)$ is surjective for all $i \ge 0$, $q \ge 1$, $\nu \in \ZZ$ and $\pp \in \Spec(A)$.
	\end{enumerate}
\end{theorem}

Quite importantly, these modules and sheaves commute with arbitrary base change in cohomology.

\begin{lemma}[{\cite[Theorem A]{FIBER_FULL}, \cite[Proposition 3.1]{FIBER_FULL_SCHEME}}]
	\label{lem_base_ch_fib_full}
	For any $A$-algebra $B$, the following statements hold: 
	\begin{enumerate}[\rm (i)]
		\item If $M$ is a finitely generated graded $R$-module that is fiber-full over $A$, then one has a natural base change isomorphism 
		$
		\HL^i(M) \otimes_{A} B  \;\xrightarrow{\; \cong \;}\; \HL^i(M \otimes_A B)
		$
		for all $i \ge 0$.
		\item If $\FF$ is a coherent $\OO_{\PP_{A}^r}$-module that is fiber-full over $A$, then one has a natural base change isomorphism 
		$
		\HH^i(\PP_{A}^r, \FF(\nu)) \otimes_{A} B \; \xrightarrow{\; \cong \;} \HH^i(\PP_{B}^r, \FF(\nu) \otimes_A B)
		$
		for all $i \ge 0$ and $\nu \in \ZZ$.
	\end{enumerate}
\end{lemma}

In \cite{FIBER_FULL_SCHEME} we introduced the natural parameter space for fiber-full sheaves. 
This new object provided a generalization of the classical Hilbert and Quot schemes. 
Here we only consider a special case, and we refer the reader to \cite{FIBER_FULL_SCHEME} for more details.

We consider the following covariant functor $\fFib_{\PP_A^r/A} : (A\text{-alg}) \rightarrow (\text{Sets})$ from the category of Noetherian $A$-algebras to the category of sets 
$$
\fFib_{\PP_A^r/A}(B) \,:= \, \Big\lbrace \text{closed subschemes $Z \subset \PP_{B}^r$} \,\mid\, \HH^i(\PP_{B}^r, \OO_{Z}(\nu)) \text{ is $B$-flat for all $i \ge 0$, $\nu \in \ZZ$} \Big\rbrace. 
$$
We stratify this functor in terms of ``Hilbert functions'' for all the cohomologies.
For a tuple of functions $\mathbf{h} = (h_0,\ldots, h_r) : \ZZ^{r+1} \rightarrow \NN^{r+1}$, we define the subfunctor 
$$
\fFib_{\PP_A^r/A}^\mathbf{h} (B) \,:= \, \left\lbrace Z \in \fFib_{\PP_A^r/A}(B)  \;
\begin{array}{|l}
	\dim_{\kappa(\pp)} \big(\HH^i(\PP_{\kappa(\pp)}^r, \OO_{Z_\pp}(\nu))\big) = h_i(\nu)\\
	\text{for all $0 \le i \le r$, $\nu \in \ZZ$, $\pp \in \Spec(B)$}		
\end{array}
  \right\rbrace,
$$
where $\kappa(\pp) := B_\pp/\pp B_\pp$ is the residue field of $\pp \in \Spec(B)$ and $Z_\pp := Z \times_{\Spec(B)} \Spec(\kappa(\pp))$ is the corresponding fiber.
The following theorem is the main result of \cite{FIBER_FULL_SCHEME} and yields the existence of the \emph{fiber-full scheme}.

\begin{theorem}[{\cite[Theorem A]{FIBER_FULL_SCHEME}}]
		Assume \autoref{setup_prelim}.
		Let $\mathbf{h} = (h_0,\ldots,h_r) : \ZZ^{r+1} \rightarrow \NN^{r+1}$ be a tuple of functions and suppose that $P_\mathbf{h} := \sum_{i = 0}^{r} {(-1)}^r h_i \in \QQ[m]$ is a numerical polynomial.
		Then, there is a quasi-projective $A$-scheme $\Fib^{\mathbf{h}}_{\PP_A^r/A}$ that represents the functor $\fFib^{\mathbf{h}}_{\PP_A^r/A}$ and that is a locally closed subscheme of the Hilbert scheme $\Hilb^{P_\mathbf{h}}_{\PP_A^r/A}$.
\end{theorem}

A direct consequence of this theorem is that the functor $\fFib_{\PP_A^r/A}$ is represented by a scheme given as the disjoint union $\Fib_{\PP_A^r/A} := \bigsqcup_{\mathbf{h} : \ZZ^{r+1} \rightarrow \NN^{r+1}} \Fib^{\mathbf{h}}_{\PP_A^r/A}$.
We refer to both $\Fib^{\mathbf{h}}_{\PP_A^r/A}$ and $\Fib_{\PP_A^r/A}$ as  fiber-full schemes.

\section{The fiber-full scheme for homogeneous ideals}
\label{sect_hom_fib_full}

In this section, we introduce a version of the fiber-full scheme that works for homogeneous ideals. 
This scheme provides a cohomological stratification of the \emph{multigraded Hilbert scheme} \cite{HAIMAN_STURMFELS} of Haiman and Sturmfels.
The multigraded Hilbert scheme was also constructed in \cite{HOM_HILB_SCH} and called the \emph{homogeneous Hilbert scheme}.
Here we continue using \autoref{setup_prelim}.
To simplify notation, for any $A$-algebra $B$, let $R_B$ be the standard graded polynomial ring $R_B:= R \otimes_{A} B$.

Following \cite{HAIMAN_STURMFELS}, given a function $g : \NN \rightarrow \NN$, we consider the covariant functor  $\fHS_{R/A}^g : (A\text{-alg}) \rightarrow (\text{Sets})$ given by 
$$\fHS_{R/A}^g(B) \,:=\, \left\lbrace I \subset R_B \text{ homogeneous ideal}\;
\begin{array}{|l}
	\left[R_B/I\right]_\nu \text{ is locally free over $B$}\\ 
	\text{of constant rank $g(\nu)$ for all $\nu \ge 0$}
\end{array}
 \right\rbrace.
$$
By \cite[Theorem 1.1]{HAIMAN_STURMFELS} (also, see \cite{HOM_HILB_SCH}), there is a projective $A$-scheme $\HS_{R/A}^g$ that represents the functor $\fHS_{R/A}^g $.
The multigraded Hilbert scheme $\HS_{R/A}^g$ is, by construction, the parameter space for homogeneous ideals with Hilbert function $g$.
We denote by $\HS_{R/A} := \bigsqcup_{g : \NN \rightarrow \NN} \HS^g_{R/A}$ the parameter space of homogeneous ideals that give a flat quotient.

We now stratify the multigraded Hilbert scheme in cohomological terms.
Given a function $g : \NN \rightarrow \NN$ and a tuple of functions $\mathbf{h} = (h_0,\ldots,h_{r+1}) : \ZZ^{r+2} \rightarrow \NN^{r+2}$, we define a subfunctor of $\fHS_{R/A}^g$ given by 
$$
\fFib_{R/A}^{g,\mathbf{h}}(B) \,:= \, \left\lbrace 
I \in \fHS_{R/A}^g(B) \; 
\begin{array}{|l}
	\left[\HL^i\left(R_B/I\right)\right]_\nu \text{ is locally free over $B$ of} \\
	\text{constant rank $h_i(\nu)$ for all $0 \le i \le r+1$, $\nu \in \ZZ$}
\end{array}
\right\rbrace.
$$
Thus, $\fFib_{R/A}^{g,\mathbf{h}}$ provides a cohomological stratification of $\fHS_{R/A}^g$.
By using the same techniques developed in \cite{FIBER_FULL_SCHEME}, we can show that the functor $\fFib_{R/A}^{g,\mathbf{h}}$ is representable.

\begin{theorem}
	\label{thm_homog_fib_sch}
	Assume \autoref{setup_prelim}.
	Let $g : \NN \rightarrow \NN$ be a function and $\mathbf{h} = (h_0,\ldots,h_{r+1}) : \ZZ^{r+2} \rightarrow \NN^{r+2}$ be a tuple of functions.
	Then, there is a quasi-projective $A$-scheme $\Fib^{g,\mathbf{h}}_{R/A}$ that represents the functor $\fFib^{g,\mathbf{h}}_{R/A}$ and that is a locally closed subscheme of the multigraded Hilbert scheme $\HS^{g}_{R/A}$.
\end{theorem} 
\begin{proof}
	The proof follows essentially along the same lines of \cite[Theorem 5.4]{FIBER_FULL_SCHEME}.
	By the existence of the multigraded Hilbert scheme \cite{HAIMAN_STURMFELS}, the functor $\fHS_{R/A}^g$ is represented by the projective $A$-scheme $\HS_{R/A}^g$.
	Since both functors $\fHS_{R/A}^g$ and $\fFib^{g,\mathbf{h}}_{R/A}$ are Zariski sheaves, we may substitute $\HS_{R/A}^g$ by an affine open subscheme $\Spec(L) \subset \HS_{R/A}^g$.
	Consider the universal ideal $I_L \subset R_L$ representing the subfunctor $\FF_L := h_{\Spec(L)}$ of $\fHS_{R/A}^g = h_{\HS_{R/A}^g}$.
	In other words, for any $A$-algebra $B$ and any $I \in \FF_L(B)$, there is a classifying $A$-homomorphism $L \rightarrow B$ such that $I \cong I_L \otimes_L B$.
	Finally, the representability is achieved by invoking \cite[Theorem 2.16]{FIBER_FULL_SCHEME}  and restricting to the locally closed subscheme $\FLoc_{R_L/I_L}^\mathbf{h} \subset \Spec(L)$.
	Indeed, $\FLoc_{R_L/I_L}^\mathbf{h}$ satisfies the condition that, for any homomorphism $L \rightarrow B$, each $\left[\HL^i(R_L/I_L \otimes_{L} B)\right]_\nu$ is locally free over $B$ of constant rank $h_i(\nu)$ if and only if $L \rightarrow B$ can be factored through $\FLoc_{R_L/I_L}^\mathbf{h}$.
\end{proof}

We say that $\Fib_{R/A}^{g,\mathbf{h}}$ is the \emph{fiber-full scheme} parametrizing homogeneous ideals with Hilbert function $g$ and cohomological data $\mathbf{h}=(h_0,\ldots,h_{r+1})$.
Then 
$$
\Fib_{R/A} \,:=\, \bigsqcup_{\substack{g : \NN \rightarrow \NN\\ \mathbf{h} : \ZZ^{r+2} \rightarrow \NN^{r+2}}} \Fib_{R/A}^{g,\mathbf{h}}
$$
is the \emph{fiber-full scheme} that parametrizes all homogeneous ideals giving a flat quotient with flat local cohomology modules.

\begin{remark}
	Let $I \subset R$ be a homogeneous ideal such that $R/I$ is $A$-flat and $\HL^i(R/I)$ is $A$-flat for all $i \ge 0$.
	Let $g : \NN \rightarrow \NN$ and $\mathbf{h} = (h_0,\ldots,h_{r+1}) : \ZZ^{r+2} \rightarrow \NN^{r+2}$, and suppose that $[I] \in \Fib_{R/A}^{g,\mathbf{h}}$.
	It then follows that $R/I$ has a Hilbert polynomial equal to $P_{g,\mathbf{h}} := g - \sum_{i=0}^{r+1}{(-1)}^ih_i$ (see \cite[Theorem 4.4.3]{BRUNS_HERZOG}).
\end{remark}

\section{Tangent-obstruction theory for the fiber-full scheme}
\label{sect_tan_obs}

In this section, we provide a tangent-obstruction theory for the fiber-full scheme.  
We compute the tangent space and determine some sufficient conditions for the existence of obstructions. 
There is a little bit of overlap as we treat both fiber-full schemes $\Fib_{\PP_\kk^r/\kk}$ (parametrizing closed subschemes) and $\Fib_{R/\kk}$ (parametrizing homogeneous ideals), and the techniques to study both are fairly similar. 
For organizational purposes, we divide the section into two subsections.
Throughout the entire section, we use the following setup. 

\begin{setup}
	\label{setup_tan_obs}
	Let $\kk$ be a field, $R = \kk[x_0,\ldots,x_r]$ be a standard graded polynomial ring, $\PP_{\kk}^r = \Proj(R)$, and $\mm = (x_0,\ldots, x_r) \subset R$ be the graded irrelevant ideal.
	Consider the parameters spaces $\Fib_{\PP_\kk^r/\kk}$ and $\Fib_{R/\kk}$ introduced in \autoref{sect_prelim} and \autoref{sect_hom_fib_full}, respectively. 
\end{setup}

Here are a couple of basic remarks that we will need.

\begin{remark}
	Given a closed subscheme $Z \subset \PP_\kk^r$, the tangent space of  $\Hilb_{\PP_{\kk}^r/\kk}$ at the corresponding point $[Z]$ is given by $T_{[Z]} \Hilb_{\PP_{\kk}^r/\kk} = \Hom_{\PP_\kk^r}(\II_{Z}, \OO_{Z})$.
	Since $\Fib_{\PP_{\kk}^r/\kk}$ is a locally closed subscheme of $\Hilb_{\PP_{\kk}^r/\kk}$, the tangent space $T_{[Z]} \Fib_{\PP_{\kk}^r/\kk}$ of $\Fib_{\PP_{\kk}^r/\kk}$ at $[Z]$ is a $\kk$-vector subspace of $\Hom_{\PP_\kk^r}(\II_{Z}, \OO_{Z})$.
\end{remark}

\begin{remark}
	Given a homogeneous ideal $I \subset R$, the tangent space of  $\HS_{R/\kk}$ at the corresponding point $[I]$ is given by $T_{[I]} \HS_{R/\kk} = \left[\Hom_{R}(I, R/I)\right]_0$.
	Since $\Fib_{R/\kk}$ is a locally closed subscheme of $\HS_{R/\kk}$, the tangent space $T_{[I]} \Fib_{R/\kk}$ of $\Fib_{R/\kk}$ at $[I]$ is a $\kk$-vector subspace of $\left[\Hom_{R}(I, R/I)\right]_0$.
\end{remark}

\begin{remark}[Tangent-obstruction theory of a functor]
	\label{rem_tan_obs}
		We recall from \cite[Chapter 6]{FGA_EXPLAINED} the definition of a tangent-obstruction theory for a deformation functor.
		Let $\FF : (\Art_\kk) \rightarrow (\text{Sets})$ be a deformation functor from the category of Artinian local $\kk$-algebras with residue field $\kk$ to the category of sets. 
		A \emph{tangent-obstruction theory} for $\FF$ is a pair of $\kk$-vector spaces $(\text{T}, \text{Ob})$ such that for every small extension $0 \rightarrow \aaa \rightarrow C' \rightarrow C \rightarrow 0$ in $\Art_\kk$ and every $u \in \FF(C)$ we have:
		\begin{enumerate}
			\item There is an element $\text{ob}(u, C') \in \text{Ob} \otimes_{\kk} \aaa$ such that $\text{ob}(u, C')=0$ if and only if there is some lifting $u' \in \FF(C')$ of $u$.
			\item If $\text{ob}(u, C')=0$, then $T \otimes_{\kk} \aaa$ acts transitively on the set of liftings of $u$ into $C'$. 
		\end{enumerate} 
\end{remark}

\subsection{The fiber-full scheme $\Fib_{R/\kk}$ for homogeneous ideals}
Given a homogeneous ideal $I_0 \subset R$, we consider the \emph{local fiber-full functor} $\fFib_{I_0} : (\Art_\kk) \rightarrow (\text{Sets})$ given by 
$$
\fFib_{I_0}(A) \,:=\, \big\{ I \in \fFib_{R/\kk}(A)  \mid I \otimes_{A} \kk \cong I_0\big\}.$$

We start by giving some preparatory results. 
We  provide a base-change lemma that will be useful, and that is part of the folklore (cf.~\cite[Exercise 6.1.10, Proposition 6.1.11]{Brodmann_Sharp_local_cohom}).

\begin{lemma}
	\label{lem_last_base_change_loc}
	Let $B$ be a Noetherian ring and $S = B[x_0,\ldots,x_r]$ be a standard graded polynomial ring. Let $M$ be a finitely generated graded $S$-module that is flat over $B$. 
	Let $c = \max\big\{ i \mid \HL^i(M) \neq 0 \big\}$.
	Then, we have a base change isomorphism 
	$$
	\HL^i(M) \otimes_B N\,  \xrightarrow{\;\cong\;} \, \HL^i(M \otimes_B N)
	$$
	for any $B$-module $N$ and for all $i \ge c$.
\end{lemma}
\begin{proof}
	Let $c'$ be an integer such that  $\HL^i(M \otimes_B N) = 0$  for all  $i \ge c'+1$ and any $B$-module $N$.
	Fix one such $B$-module $N$.
	Pick a free $B$-presentation $F_1 \rightarrow F_0 \rightarrow N \rightarrow 0$ of $N$, and let $K = \IM(F_1 \rightarrow F_0)$ be the corresponding module of syzygies.
	The $B$-flatness of $M$ yields a short exact sequence $0 \rightarrow M \otimes_B K \rightarrow M \otimes_B F_0 \rightarrow M \otimes_B N \rightarrow 0$, and so we get an exact sequence in cohomology
	$$
	\HL^{c'}(M \otimes_B K) \rightarrow \HL^{c'}(M \otimes_B F_0) \rightarrow \HL^{c'}(M \otimes_B N) \rightarrow \HL^{c'+1}(M \otimes_B K) = 0.
	$$
	This gives the following exact sequence 
	$$
	\HL^{c'}(M) \otimes_B F_1 \rightarrow \HL^{c'}(M) \otimes_B F_0  \rightarrow \HL^{c'}(M \otimes_B N) \rightarrow 0,
	$$
	and as a consequence, the isomorphism $\HL^{c'}(M) \otimes_B N \xrightarrow{\cong} \HL^{c'}(M \otimes_B N)$. 
	
	It is clear that $c \le c'$.
	On the other hand, the already proved isomorphism $\HL^{c'}(M) \otimes_B N \xrightarrow{\cong} \HL^{c'}(M \otimes_B N)$ shows that $c \ge \min\big\{ c' \mid \text{$\HL^i(M \otimes_B N) = 0$  for all  $i \ge c'+1$ and any $B$-module $N$} \big\}$.
	Therefore, the result of the lemma holds.
\end{proof}

The following  lemma gives another equivalent characterization of fiber-fullness.

\begin{lemma}
	\label{lem_fib_full_crit_loc}
	Let $(B, \bb, \kk)$ be a Noetherian local ring and $S = B[x_0,\ldots,x_r]$ be a standard graded polynomial ring.
	Let $I \subset S$ be a homogeneous ideal such that $S/I$ is flat over $B$, and $I_0 = I \otimes_{B} \kk \subset R$ be the homogeneous ideal corresponding with the closed fiber.
	Then, $S/I$ is fiber-full over $B$ if and only if the natural map 
	$$
	\HL^i\left(I\right) \, \longrightarrow \, \HL^i\left(I_0\right)
	$$
	is surjective for all $i \ge 0$.
\end{lemma}
\begin{proof}
	By the local criterion for fiber-fullness (see \autoref{thm_fib_full_mod}), $S/I$ is fiber-full over $B$ if and only if $\HL^i(S/I \otimes_{B} B/\bb^q)  \rightarrow  \HL^i(R/I_0)$ is surjective for all $i \ge 0, q \ge 1$.
	The base change property of fiber-full modules (see \autoref{lem_base_ch_fib_full}) implies that the latter condition is equivalent to the surjectivity of the map $\HL^i(S/I)  \rightarrow  \HL^i(R/I_0)$ for all $i \ge 0$.
	
	First, suppose that $I_0 = 0$. 
	Since $S/I$ is $B$-flat,   $I = 0$ and the fiber-fullness condition holds.
	
	Next, we assume that $I_0 \neq 0$.
	By \autoref{lem_last_base_change_loc}, $\HL^{r+1}(S/I) \otimes_B \kk \xrightarrow{\cong} \HL^{r+1}(R/I_0)=0$, and then Nakayama's lemma implies $\HL^{r+1}(S/I) = 0$.
	Thus, \autoref{lem_last_base_change_loc} gives the surjectivity of the map $\HL^{r}(S/I)  \rightarrow  \HL^{r}(R/I_0)$.
	From the short exact sequences $0 \rightarrow I \rightarrow  S \rightarrow S/I \rightarrow 0$ and $0 \rightarrow I_0 \rightarrow  R \rightarrow R/I_0 \rightarrow 0$, we obtain the short exact sequences $0 \rightarrow \HL^{r}(S/I) \rightarrow \HL^{r+1}(I) \rightarrow \HL^{r+1}(S) \rightarrow 0$ and $0 \rightarrow \HL^{r}(R/I_0) \rightarrow \HL^{r+1}(I_0) \rightarrow \HL^{r+1}(R) \rightarrow 0$ and the isomorphisms $\HL^{i}(I_0) \cong \HL^{i-1}(R/I_0)$ and $\HL^{i}(I) \cong \HL^{i-1}(S/I)$ for all $i \le r$.
	Therefore, $\HL^i(S/I)  \rightarrow  \HL^i(R/I_0)$ is surjective for all $i \ge 0$ if and only if $\HL^i(I)  \rightarrow  \HL^i(I_0)$ is surjective for all $i \ge 0$.
	This concludes the proof of the lemma.
\end{proof}

The following theorem determines the tangent space of the fiber-full scheme $\Fib_{R/\kk}$.

\begin{theorem} 
	\label{thm_tan_space_loc}
	Assume \autoref{setup_tan_obs}.
	Let $I_0 \subset R$ be a homogeneous ideal. 
	The tangent space of $\Fib_{R/\kk}$ at the corresponding point $[I_0]$ is given by 
	$$
	T_{[I_0]} \Fib_{R/\kk} = \big\{\varphi \in \big[\Hom(I_0,R/I_0)\big]_0 \,\mid\, \HH^i_{\mm}(\varphi) = 0 \text{ for all $i \ge 0$}\big\},
	$$
	where $\HL^i(\varphi)$ denotes the natural map $\HL^i(\varphi) : \HL^i(I_0) \rightarrow \HL^i(R/I_0)$  induced in cohomology.
\end{theorem}
\begin{proof} 
Let $\varphi \in T_{[I_0]} \HS_{R/\kk} = \left[\Hom(I_0,R/I_0)\right]_0$. 
It is known that $\varphi$ corresponds to a flat extension of $I_0$ over $\kk[t]/(t^2)$, i.e., the ideal 
$$
I = \left\{ f+ tg \mid f \in I_0 \text{ and } \varphi(f) = \bar{g} \in R/I_0 \right\} \,\subseteq\, R' := R[t]/(t^2).
$$
We need to find a condition that characterizes when $R'/I$ is fiber-full over $\kk[t]/(t^2)$.
By \autoref{lem_fib_full_crit_loc}, it suffices to determine when the natural maps $\HL^i(I) \rightarrow \HL^i(I_0)$ are surjective for all $i \ge 0$.

There is an exact sequence of $R$-modules
$
0 \to I_0 \longrightarrow I \longrightarrow I_0 \to 0
$
where the first map is multiplication by $t$.
 By using $\varphi$ we obtain a map between short exact sequences
\begin{equation*}		
	\begin{tikzpicture}[baseline=(current  bounding  box.center)]
		\matrix (m) [matrix of math nodes,row sep=3.7em,column sep=5.5em,minimum width=2em, text height=1.5ex, text depth=0.25ex]
		{
			0 & I_0 & I &  I_0 & 0  \\
			0 & I_0 & R & R/I_0 & 0. \\
		};						
		\path[-stealth]
		(m-1-1) edge node [above] {} (m-1-2)
		(m-1-2) edge node [above] {$t$} (m-1-3)
		(m-1-3) edge node [above] {} (m-1-4)
		(m-1-4) edge node [above] {} (m-1-5)
		(m-2-1) edge node [above] {} (m-2-2)
		(m-2-2) edge node [above] {} (m-2-3)
		(m-2-3) edge node [above] {} (m-2-4)
		(m-2-4) edge node [above] {} (m-2-5)
		(m-1-2) edge node [right] {=} (m-2-2)
		(m-1-3) edge node [right] {$\varphi'$} (m-2-3)
		(m-1-4) edge node [right] {$\varphi$} (m-2-4)
		;
	\end{tikzpicture}	
\end{equation*}	
Note that this diagram is commutative and $\varphi'$ is given by mapping $f+ t g \mapsto g$.
Taking the associated long exact sequence in cohomology we obtain
\begin{equation*}		
	\begin{tikzpicture}[baseline=(current  bounding  box.center)]
		\matrix (m) [matrix of math nodes,row sep=3.7em,column sep=5.5em,minimum width=2em, text height=1.5ex, text depth=0.25ex]
		{
			\cdots & \HH^i_{\mm}(I) & \HH^i_{\mm}(I_0) &  \HH^{i+1}_{\mm}(I_0) & \cdots  \\
			\cdots & \HH^i_{\mm}(R) & \HH^i_{\mm}(R/I_0) & \HH^{i+1}_{\mm}(I_0) & \cdots. \\
		};						
		\path[-stealth]
		(m-1-1) edge node [above] {} (m-1-2)
		(m-1-2) edge node [above] {} (m-1-3)
		(m-1-3) edge node [above] {} (m-1-4)
		(m-1-4) edge node [above] {} (m-1-5)
		(m-2-1) edge node [above] {} (m-2-2)
		(m-2-2) edge node [above] {} (m-2-3)
		(m-2-3) edge node [above] {} (m-2-4)
		(m-2-4) edge node [above] {} (m-2-5)
		(m-1-2) edge node [right] {} (m-2-2)
		(m-1-3) edge node [right] {$H^i_{\mm}(\varphi)$} (m-2-3)
		(m-1-4) edge node [right] {=} (m-2-4)
		;
	\end{tikzpicture}	
\end{equation*}	
Note that the surjectivity of $\HL^{r+1}(I) \rightarrow \HL^{r+1}(I_0)$ comes for free.
For each $i \le r$, the fact that $\HL^i(R) = 0$ gives the following commutative diagram
$$
\begin{tikzcd}
			\HH^i_{\mm}(I) \arrow[r] \arrow[d] 
				& \HH^i_{\mm}(I_0) \arrow[r] \arrow[d, "\HH^i_{\mm}(\varphi)"]
					&  \HH^{i+1}_{\mm}(I_0) \arrow[d, "="]  \\
			0 \arrow[r] 
				& \HH^i_{\mm}(R/I_0) \arrow[r, hook]
					& \HH^{i+1}_{\mm}(I_0). 
\end{tikzcd}
$$
It then follows that $\HL^i(I) \rightarrow \HL^i(I_0)$ is surjective for all $i \ge 0$ if and only if $\HH^i_{\mm}(\varphi) = 0$ for all $i \ge 0$. 
So, the proof of the theorem is complete.
\end{proof}

We now proceed to study the question of lifting fiber-full deformations over Artinian rings. 
The following result gives an equivalent of \cite[Theorem 6.2]{hartshorne2010deformation} in the setting of the fiber-full scheme $\Fib_{R/\kk}$.

Let $(C', \nn',\kk)$ and $(C, \nn,\kk)$ be Artinian local rings with residue field $\kk$, and suppose we have a short exact sequence 
$$
0 \rightarrow \aaa \rightarrow C' \rightarrow C \rightarrow 0
$$
where $\nn' \aaa = 0$ (this implies, in particular, that $\aaa$ can be considered as a $\kk$-vector space).
Let $S$ and $S'$ be the standard graded polynomial rings $S = C[x_0,\ldots,x_r]$ and $S' = C'[x_0,\ldots,x_r]$.
Let $I \subset S$ be a homogeneous ideal such that $S/I$ is fiber-full over $C$, and denote by $I_0 \subset R$  the homogeneous ideal $I_0 = I \otimes_{C} \kk$ (in other words, $I$ is a fiber-full extension of $I_0$ over $C$).
Now, we wish to classify all the fiber-full extensions of $I$ over $C'$, that is, we seek all homogeneous ideals $I' \subset S'$ such that $S'/I'$ is fiber-full over $C'$ and $I \cong I' \otimes_{C'} C$.
The following theorem achieves this goal. 
In particular, it implies that the local fiber-full functor $\fFib_{I_0} : (\Art_\kk) \rightarrow (\text{Sets})$ satisfies condition (2) of \autoref{rem_tan_obs}.

\begin{theorem}
	\label{thm_small_extensions_loc}
	Assume \autoref{setup_tan_obs} and the notations above. 
	The set of fiber-full extensions of $[I] \in \Fib_{S/C}$ over $C'$
	$$
	\left\lbrace [I'] \in \Fib_{S'/C'} \;\text{ such that }\;  I' \otimes_{C'} C \cong I \right\rbrace
	$$
	is a pseudotorsor under the action of the subgroup of $\big[\Hom_R\left( I_0, R/I_0 \otimes_\kk \aaa \right)\big]_0$ given by
	$$
	\big\lbrace \varphi \in \big[\Hom_R\left( I_0, R/I_0 \otimes_\kk \aaa \right)\big]_0 \;\mid\; \HL^i(\varphi)=0 \text{ for all $i \ge 0$} \big\rbrace, 
	$$
	where $\HL^i(\varphi)$ denotes the natural map $\HL^i(\varphi) : \HL^i(I_0) \rightarrow \HL^i(R/I_0 \otimes_\kk \aaa)$ induced in cohomology.
\end{theorem}
\begin{proof}
	First, we recall that the set of all flat extensions of $I \subset S$ over $C'$ 
	$$
	\left\lbrace [I'] \in \HS_{S'/C'} \;\text{ such that }\;  I' \otimes_{C'} C \cong I \right\rbrace
	$$
	is a pseudotorsor under the action of the group of $\big[\Hom_R\left( I_0, R/I_0 \otimes_\kk \aaa \right)\big]_0$ (see, e.g., \cite[Theorem 6.2]{hartshorne2010deformation}).
	
	Let $\left[I_1'\right], \left[I_2'\right] \in \HS_{S'/C'}$ be two possible flat extensions of $I \subset S$ over $C'$.
	We then get the following short exact sequences 
	$
	0 \rightarrow I \otimes_C \aaa \xrightarrow{\iota_1} I_1' \xrightarrow{\pi_1} I \rightarrow 0
	$
	 and 
	$0 \rightarrow I \otimes_C \aaa \xrightarrow{\iota_2} I_2' \xrightarrow{\pi_2} I \rightarrow 0.
	$
	We consider the diagonal $S'$-module given by 
 $\Delta := \big\lbrace (x_1,x_2,x) \in  I_1' \oplus I_2' \oplus I  \, \mid \, \pi_1(x_1) = x = \pi_2(x_2) \big\rbrace$.
We have the following commutative diagram
	\begin{equation*}		
		\begin{tikzpicture}[baseline=(current  bounding  box.center)]
			\matrix (m) [matrix of math nodes,row sep=3.5em,column sep=3.5em,minimum width=2em, text height=1.5ex, text depth=0.25ex]
			{
				0 & \left(I \otimes_C \aaa\right) \oplus \left(I \otimes_C \aaa\right)  & \Delta &  I & 0  \\
				0 & I \otimes_C \aaa & S \otimes_C \aaa &  S/I \otimes_C \aaa & 0,  \\
			};						
			\path[-stealth]
			(m-1-1) edge node [above] {} (m-1-2)
			(m-1-2) edge node [above] {$(\iota_1,\iota_2,0)$} (m-1-3)
			(m-1-3) edge node [above] {$\pi$} (m-1-4)
			(m-1-4) edge node [above] {} (m-1-5)
			(m-2-1) edge node [above] {} (m-2-2)
			(m-2-2) edge node [above] {} (m-2-3)
			(m-2-3) edge node [above] {} (m-2-4)
			(m-2-4) edge node [above] {} (m-2-5)
			(m-1-2) edge node [right] {} (m-2-2)
			(m-1-3) edge node [right] {$\varphi'$} (m-2-3)
			(m-1-4) edge node [right] {$\varphi$} (m-2-4)
			;
		\end{tikzpicture}	
	\end{equation*}	
	where $\pi : \Delta \subset I_1' \oplus I_2' \oplus I \rightarrow I$ is the natural projection and $\varphi' : \Delta \rightarrow  \OO_{\PP_{C'}^r}$ is induced by the difference $\gamma_1 - \gamma_2$ of the natural injections $\gamma_1 : I_1' \hookrightarrow S'$ and $\gamma_2 : I_2' \hookrightarrow S'$.
	Indeed, there is a short exact sequence 
	$0 \rightarrow S \otimes_C \aaa \rightarrow S' \rightarrow S \rightarrow 0$, and since the image of $\varphi'$ goes to zero in $S$, we get the claimed map $\varphi' : \Delta \rightarrow S \otimes_C \aaa$. 
	Then $\varphi'$ induces the map $\varphi : I \rightarrow S/I \otimes_{C} \aaa$. 
	
	Recall that, after fixing $I_1' \subset S'$, it is equivalent to have the other extension $I_2' \subset S'$ or the map $\varphi$.
	
	Fix $0 \le i \le r$.
	From the above diagram, we obtain the following diagram with exact rows
	\begin{equation*}		
		\begin{tikzpicture}[baseline=(current  bounding  box.center)]
			\matrix (m) [matrix of math nodes,row sep=3.5em,column sep=1.2em,minimum width=2em, text height=1.5ex, text depth=0.25ex]
			{
				\HL^i\big(\left(I \otimes_C \aaa\right)^2\big) &  \HL^i(\Delta)  & \HL^i(I) &  \HL^{i+1}\big(\left(I \otimes_C \aaa\right)^2 \big) & \HL^{i+1}(\Delta) \\
				\HL^i\left(I \otimes_C \aaa\right) &  \HL^i(S \otimes_C \aaa) =0  & \HL^i(S/I \otimes_C \aaa)  & \HL^{i+1}\left(I \otimes_C \aaa\right).   \\
			};						
			\path[-stealth]
			(m-1-1) edge node [above] {} (m-1-2)
			(m-1-2) edge node [above] {} (m-1-3)
			(m-1-3) edge node [above] {$\Psi$} (m-1-4)
			(m-1-4) edge node [above] {$\Phi$} (m-1-5)
			(m-2-1) edge node [above] {} (m-2-2)
			(m-2-2) edge node [above] {} (m-2-3)
			(m-2-3) edge node [above] {} (m-2-4)
			(m-1-1) edge node [right] {} (m-2-1)
			(m-1-2) edge node [right] {} (m-2-2)
			(m-1-3) edge node [right] {$\HL^i(\varphi)$} (m-2-3)
			(m-1-4) edge node [right] {} (m-2-4)
			;
		\end{tikzpicture}	
	\end{equation*}	
	Now assume that $S'/I_1'$ is fiber-full over $C'$.
	It follows that  $\HL^i(\pi_1) : \HL^i(I_1') \rightarrow \HL^i(I)$ is surjective, and so from the  sequence in cohomology induced by 0 $\rightarrow I \otimes_C \aaa \xrightarrow{\iota_1} I_1' \xrightarrow{\pi_1} I \rightarrow 0$, we have that the induced map $\HL^{i+1}(\iota_1): \HL^{i+1}(I \otimes_C \aaa) \rightarrow \HL^{i+1}(I_1')$ is injective.
	Hence, we get the commutative diagram 
	\begin{equation*}		
		\begin{tikzpicture}[baseline=(current  bounding  box.center)]
			\matrix (m) [matrix of math nodes,row sep=3.5em,column sep=3em,minimum width=2em, text height=1.5ex, text depth=0.25ex]
			{
				\HL^{i+1}(I \otimes_C \aaa) &  & \HL^{i+1}(I_1')   \\
				\HL^{i+1}\big((I \otimes_C \aaa)^2\big) &  \HL^{i+1}(\Delta)  & \HL^{i+1}(I_1') \oplus \HL^{i+1}(I_2').   \\
			};						
			\path[-stealth]
			(m-2-1) edge node [above] {$\Phi$} (m-2-2)
			(m-2-2) edge (m-2-3)
			(m-2-3) edge (m-1-3)
			;
			\draw[right hook->] (m-1-1)--(m-1-3) node [midway,above] {$\HL^{i+1}(\iota_1)$};
			\draw[right hook->] (m-1-1)--(m-2-1);
		\end{tikzpicture}	
	\end{equation*}	
	As a consequence, we have $\IM(\Psi) = \Ker(\Phi) \subset 0 \,\oplus \, \HL^{i+1}(I \otimes_C \aaa) \subset \HL^{i+1}\big((I \otimes_C \aaa)^2\big)$.
	So, we obtain the next simplified diagram 
	\begin{equation*}		
		\begin{tikzpicture}[baseline=(current  bounding  box.center)]
			\matrix (m) [matrix of math nodes,row sep=3.5em,column sep=3.2em,minimum width=2em, text height=1.5ex, text depth=0.25ex]
			{
				\HL^i(\Delta)  & \HL^i(I) &  0\, \oplus\, \HL^{i+1}\left(I \otimes_C \aaa\right)   \\
				0 & \HL^i(S/I \otimes_C \aaa)  & \HL^{i+1}\left(I \otimes_C \aaa\right).   \\
			};						
			\path[-stealth]
			(m-1-1) edge (m-1-2)
			(m-1-1) edge (m-2-1)
			(m-2-1) edge (m-2-2)
			(m-1-2) edge node [above] {$\Psi$} (m-1-3)
			(m-1-3) edge node [right] {$\cong$} (m-2-3)
			(m-1-2) edge node [right] {$\HL^i(\varphi)$} (m-2-2)
			;
			\draw[right hook->] (m-2-2)--(m-2-3);	
		\end{tikzpicture}	
	\end{equation*}	
	It now follows that $\HL^i(\varphi)=0$ if and only if the map $\HL^i(\Delta)  \rightarrow \HL^i(I)$ is surjective.
	
	We claim that $\HL^i(\Delta)  \rightarrow \HL^i(I)$ is surjective if and only if $\HL^i(\pi_2): \HL^i(I_2')  \rightarrow \HL^i(I)$ is surjective.
	Since $\pi : \Delta \rightarrow I$ factors as $\Delta \rightarrow I_2' \xrightarrow{\pi_2} I$, it follows that $\HL^i(\pi_2)$ is surjective when $\HL^i(\Delta)  \rightarrow \HL^i(I)$ is.
	On the other hand, the injection $\iota_1 \oplus \iota_2 : \left(I \otimes_C \aaa\right) \oplus \left(I \otimes_C \aaa\right) \hookrightarrow I_1' \oplus I_2'$ factors as 
	$$
	\left(I \otimes_C \aaa\right) \oplus \left(I \otimes_C \aaa\right)  \xrightarrow{(\iota_1,\iota_2,0)}  \Delta \rightarrow I_1' \oplus I_2', 
	$$
	and so the injectivity of  $\HL^{i+1}(\iota_1)$ and $\HL^{i+1}(\iota_2)$ implies the injectivity of $\Phi : \HL^{i+1}\big((I \otimes_C \aaa)^2\big) \rightarrow  \HL^{i+1}(\Delta)$. 
	This shows the claim. 
	
	Finally, under assumption that $S'/I_1'$ is fiber-full over $C'$, we have shown that $\HL^i(\pi_2): \HL^i(I_2')  \rightarrow \HL^i(I)$ is surjective for all $i \ge 0$ if and only if $\HL^i(\varphi)=0$ for all $i \ge 0$.
	Note that by \autoref{lem_fib_full_crit_loc}, $\HL^i(I) \rightarrow \HL^i(I_0)$ is surjective for all $i \ge 0$.
	So, the result of the theorem now follows from \autoref{lem_fib_full_crit_loc}.
\end{proof}

We  concentrate on the obstructions of the fiber-full scheme $\Fib_{R/\kk}$.
We determine sufficient conditions for the existence of obstructions. 
So, for the local fiber-full functor $\fFib_{I_0} : (\Art_\kk) \rightarrow (\text{Sets})$, we only obtain a partial answer for part (1) of \autoref{rem_tan_obs}.

\begin{theorem}
	\label{thm_obs_loc}
	Assume \autoref{setup_tan_obs}.
	Let $I_0 \subset R$ be a homogeneous ideal. 
	Let $0 \rightarrow \aaa \rightarrow C' \rightarrow C \rightarrow 0$ be a small extension in $\Art_\kk$ and $I \in \fFib_{I_0}(C)$.
	There are elements $\ob(I, C') \in \big[\Ext_R^1(I_0, R/I_0 \otimes_{\kk} \aaa)\big]_0$ and $\ob_i(I, C') \in \big[\Ext_R^2\left(\HL^i(R/I_0), \HL^i(R/I_0) \otimes_{\kk} \aaa\right)\big]_0$ for all $0 \le i \le r+1$ such that, if 
	$$
	\ob(I, C') \neq 0 \quad \text{ or } \quad \ob_i(I, C') \neq 0 \text{ for some } 0 \le i \le r+1,
	$$
	then there is no lifting $I' \in \fFib_{I_0}(C')$ of $I$.
\end{theorem}
\begin{proof}
	There is an element $\ob(I, C') \in \big[\Ext_R^1(I_0, R/I_0 \otimes_{\kk} \aaa)\big]_0$ such that $\ob(I, C') = 0$ if and only if there is a flat extension of $I$ into $C'$ (see, e.g., \cite[\S I.2]{RAT_CURVES_KOLLAR}).
	If $I' \in \fFib_{I_0}(C')$ is a lifting of $I$, then \autoref{lem_base_ch_fib_full} gives short exact sequences $0 \rightarrow \HL^i(S/I) \otimes_{\kk} \aaa \rightarrow \HL^i(S'/I') \rightarrow \HL^i(S/I) \rightarrow 0$; this means that $\HL^i(S'/I')$ is a lifting of $\HL^i(S/I)$.
	For each $0 \le i \le r+1$, there is an element $\ob_i(I, C') \in \big[\Ext_R^2\left(\HL^i(R/I_0), \HL^i(R/I_0) \otimes_{\kk} \aaa\right)\big]_0$ such that $\ob_i(I, C')=0$ if and only if there is a lifting of $\HL^i(S/I)$ (see, e.g., \cite[Exercise 7.4]{hartshorne2010deformation}, \cite[Proposition 4.2]{DEFORM_RUNAR}, \cite[\href{https://stacks.math.columbia.edu/tag/08VR}{Tag 08VR}]{stacks-project}).
	So, the proof of the theorem is complete.
\end{proof}

\begin{remark}
	\label{remark_quest_obs_loc}
	We expect that  \autoref{thm_obs_loc}  does indeed provide an obstruction space for the local fiber-full functor $\fFib_{I_0}$.
\end{remark}

Now, we discuss the pro-representability of the local fiber-full functor $\fFib_{I_0} : (\Art_\kk) \rightarrow (\text{Sets})$.

\begin{remark}
	The pro-representability of the local functor $\fFib_{I_0}$ follows directly from the representability of the global functor $\fFib_{R/\kk}$ (see, e.g., \cite[Proposition 23.3]{hartshorne2010deformation}).
	Alternatively, one may check directly Schlessinger's criterion (see \cite[Theorem 16.4]{hartshorne2010deformation}), where the most difficult condition to verify $(H_4)$ is already given to us by \autoref{thm_small_extensions_loc}.
\end{remark}

\subsection{The fiber-full scheme $\Fib_{\PP_{\kk}^r/\kk}$ for closed subschemes}
\label{subsect_fib_full_sch}
Given a closed subscheme $Z_0 \subset \PP_{\kk}^r$, we consider the \emph{local fiber-full functor} $\fFib_{Z_0} : (\Art_\kk) \rightarrow (\text{Sets})$ given by 
$$
\fFib_{Z_0}(A) \,:=\, \big\{ Z \in \fFib_{\PP_{\kk}^r/\kk}(A)  \,\mid\, Z \times_{\Spec(A)} \Spec(\kk) \cong Z_0\big\}.$$

We begin with a lemma for sheaf cohomology that is the equivalent of \autoref{lem_last_base_change_loc}.

\begin{lemma}
	\label{lem_last_base_change}
	Let $B$ be a Noetherian ring and $\FF$ be a coherent sheaf on $\PP_B^r$ that is flat over $B$. 
	Let $c = \max\big\{ i \mid \HH^i(\PP_B^r, \FF) \neq 0 \big\}$.
	Then, we have a base change isomorphism 
	$$
	\HH^i(\PP_B^r, \FF) \otimes_B N\,  \xrightarrow{\;\cong\;} \, \HH^i(\PP_B^r, \FF \otimes_B N)
 	$$
 	for any $B$-module $N$ and for all $i \ge c$. 
\end{lemma}
\begin{proof}
	The proof follows along the same lines of \autoref{lem_last_base_change_loc}.
\end{proof}

The next lemma gives another equivalent description of fiber-full schemes.

\begin{lemma}
	\label{lem_fib_full_crit}
	Let $(B, \bb, \kk)$ be a Noetherian local ring.
	Let $Z \subset \PP_B^r$ be a closed subscheme that is flat over $B$, and $Z_0 = Z \times_{\Spec(B)} \Spec(\kk) \subset \PP_\kk^r$ be the closed subscheme corresponding with the closed fiber.
	Then, $Z$ is fiber-full over $B$ if and only if the following natural map 
	$$
	\HH^i\left(\PP_B^r, \II_{Z}(\nu)\right) \, \longrightarrow \, \HH^i\left(\PP_\kk^r, \II_{Z_0}(\nu)\right)
	$$
	is surjective for all $i \ge 1$ and $\nu \in \ZZ$.
\end{lemma}
\begin{proof}
	By the local criterion for fiber-fullness (see \autoref{thm_fib_full_shv}), $Z$ is fiber-full over $B$ is and only if $\HH^i(\PP_{B_q}^r, \OO_{Z_q}(\nu))  \rightarrow  \HH^i(\PP_\kk^r, \OO_{Z_0}(\nu))$ is surjective for all $i \ge 0, q \ge 1, \nu \in \ZZ$, where $B_q = B/\bb^q$ and $Z_q = Z \times_{\Spec(B)} \Spec(B_q) \subset \PP_{B_q}^r$.
	The base change property of fiber-full sheaves (see \autoref{lem_base_ch_fib_full}) implies that the latter condition is equivalent to the surjectivity of the map $\HH^i(\PP_B^r, \OO_{Z}(\nu))  \rightarrow  \HH^i(\PP_\kk^r, \OO_{Z_0}(\nu))$ for all $i \ge 0, \nu \in \ZZ$.
	
	First, suppose that $Z_0 = \PP_{\kk}^r$. 
	Since $Z$ is $B$-flat,  $Z = \PP_{B}^r$ and the fiber-fullness condition holds.
	
	Next, we assume that $Z_0 \subsetneq \PP_\kk^r$.
	Let $S = B[x_0,\ldots,x_r]$,  $I_Z = \bigoplus_{\nu \in \ZZ} \HH^0(\PP_{B}^r, \II_Z(\nu)) \subset S$ and $I_{Z_0} = \bigoplus_{\nu \in \ZZ} \HH^0(\PP_{\kk}^r, \II_{Z_0}(\nu)) \subset R$.
	By \autoref{lem_last_base_change}, $\HH^r(\PP_{B}^r, \OO_{Z}(\nu)) \otimes_B \kk \xrightarrow{\cong} \HH^r(\PP_{\kk}^r, \OO_{Z_0}(\nu))=0$, and then Nakayama's lemma implies $\HH^r(\PP_{B}^r, \OO_{Z}(\nu)) = 0$.
	Thus, \autoref{lem_last_base_change} gives the surjectivity of the map $\HH^{r-1}(\PP_B^r, \OO_{Z}(\nu))  \rightarrow  \HH^{r-1}(\PP_\kk^r, \OO_{Z_0}(\nu))$.
	From the short exact sequences $0 \rightarrow \II_Z \rightarrow  \OO_{\PP_{B}^r} \rightarrow \OO_{Z} \rightarrow 0$ and $0 \rightarrow \II_{Z_0} \rightarrow  \OO_{\PP_{\kk}^r} \rightarrow \OO_{Z_0} \rightarrow 0$, we obtain the following commutative diagrams
	\begin{equation*}		
		\begin{tikzpicture}[baseline=(current  bounding  box.center)]
			\matrix (m) [matrix of math nodes,row sep=3.5em,column sep=2.5em,minimum width=2em, text height=1.5ex, text depth=0.25ex]
			{
				0 & S/I_Z & \bigoplus_{\nu \in \ZZ} \HH^0(\PP_{B}^r, \OO_{Z}(\nu)) &  \bigoplus_{\nu \in \ZZ} \HH^1(\PP_{B}^r, \II_{Z}(\nu)) & 0  \\
				0 & R/I_{Z_0} & \bigoplus_{\nu \in \ZZ} \HH^0(\PP_{\kk}^r, \OO_{Z_0}(\nu)) &  \bigoplus_{\nu \in \ZZ} \HH^1(\PP_{\kk}^r, \II_{Z_0}(\nu)) & 0  \\
			};						
			\path[-stealth]
			(m-1-1) edge node [above] {} (m-1-2)
			(m-1-2) edge node [above] {} (m-1-3)
			(m-1-3) edge node [above] {} (m-1-4)
			(m-1-4) edge node [above] {} (m-1-5)
			(m-2-1) edge node [above] {} (m-2-2)
			(m-2-2) edge node [above] {} (m-2-3)
			(m-2-3) edge node [above] {} (m-2-4)
			(m-2-4) edge node [above] {} (m-2-5)
			(m-1-3) edge node [right] {} (m-2-3)
			(m-1-4) edge node [right] {} (m-2-4)
			;
			\path [draw,->>] (m-1-2) -- (m-2-2);	
		\end{tikzpicture}	
	\end{equation*}	
	and 
		\begin{equation*}		
		\begin{tikzpicture}[baseline=(current  bounding  box.center)]
			\matrix (m) [matrix of math nodes,row sep=3.5em,column sep=3em,minimum width=2em, text height=1.5ex, text depth=0.25ex]
			{
				 \bigoplus_{\nu \in \ZZ} \HH^i(\PP_{B}^r, \OO_{Z}(\nu)) &  \bigoplus_{\nu \in \ZZ} \HH^{i+1}(\PP_{B}^r, \II_{Z}(\nu))   \\
				 \bigoplus_{\nu \in \ZZ} \HH^i(\PP_{\kk}^r, \OO_{Z_0}(\nu)) &  \bigoplus_{\nu \in \ZZ} \HH^{i+1}(\PP_{\kk}^r, \II_{Z_0}(\nu))   \\
			};						
			\path[-stealth]
			(m-1-1) edge node [above] {$\cong$} (m-1-2)
			(m-2-1) edge node [above] {$\cong$} (m-2-2)
			(m-1-1) edge (m-2-1)
			(m-1-2) edge (m-2-2)
						;
		\end{tikzpicture}	
	\end{equation*}	
	for all $1 \le i \le r-2$. 
	Therefore, $\HH^i(\PP_B^r, \OO_{Z}(\nu))  \rightarrow  \HH^i(\PP_\kk^r, \OO_{Z_0}(\nu))$ is surjective for all $i \ge 0, \nu \in \ZZ$ if and only if $\HH^i(\PP_B^r, \II_{Z}(\nu))  \rightarrow  \HH^i(\PP_\kk^r, \II_{Z_0}(\nu))$ is surjective for all $i \ge 1, \nu \in \ZZ$.
	This concludes the proof of the lemma.
\end{proof}

The next theorem explicitly describes the tangent space of the fiber-full scheme $\Fib_{\PP_{\kk}^r/\kk}$.

\begin{theorem} \label{thm_tan_space}
	Assume \autoref{setup_tan_obs}.
	Let $Z_0 \subset \PP_\kk^r$ be a closed subscheme. 
	The tangent space of $\Fib_{\PP_{\kk}^r/\kk}$ at the corresponding point $[Z_0]$ is given by 
	$$
	T_{[Z_0]} \Fib_{\PP_\kk^r/\kk} = \left\lbrace \varphi \in \Hom_{\PP_\kk^r}\left(\II_{Z_0}, \OO_{Z_0}\right)  \;\mid\; \HH^i(\PP_\kk^r, \varphi(\nu))=0 \text{ for all $i \ge 1$, $\nu \in \ZZ$} \right\rbrace,
	$$
	where $\HH^i\left(\PP_\kk^r, \varphi(\nu)\right)$ denotes the natural map $\HH^i\left(\PP_\kk^r, \varphi(\nu)\right) : \HH^i\left(\PP_\kk^r, \II_{Z_0}(\nu)\right) \rightarrow \HH^i\left(\PP_\kk^r, \left(\OO_{Z_0}\right)(\nu)\right)$ induced in cohomology.
\end{theorem}
\begin{proof}
	The proof follows similarly the one of \autoref{thm_tan_space_loc}, but now one uses \autoref{lem_fib_full_crit} instead of \autoref{lem_fib_full_crit_loc}.
\end{proof}

We  provide a result that is the equivalent of \autoref{thm_small_extensions_loc} for closed subschemes.
It shows, in particular,  that the local fiber-full functor $\fFib_{Z_0} : (\Art_\kk) \rightarrow (\text{Sets})$ satisfies condition (2) of \autoref{rem_tan_obs}.
It also extends \cite[Theorem 6.2]{hartshorne2010deformation} into the setting of the fiber-full scheme $\Fib_{\PP_{\kk}^r/\kk}$.

Let $(C', \nn')$ and $(C, \nn)$ be Artinian local rings with residue field $\kk$, and suppose we have a short exact sequence 
$$
0 \rightarrow \aaa \rightarrow C' \rightarrow C \rightarrow 0
$$
where $\nn' \aaa = 0$ (this implies, in particular, that $\aaa$ can be considered as a $\kk$-vector space).
Let $Z \subset \PP_C^r$ be a fiber-full schemer over $C$, and denote by $Z_0$  the closed subscheme $Z_0 = Z \times_{\Spec(C)} \Spec(\kk) \subset \PP_\kk^r$ (in other words, $Z$ is a fiber-full extension of $Z_0$ over $C$).
Now, we wish to classify all the fiber-full extensions of $Z$ over $C'$, that is, we seek all the closed subschemes $Z' \subset \PP_{C'}^r$ such that $Z'$ is fiber-full over $C'$ and $Z \cong Z' \times_{\Spec(C')} \Spec(C)$.

\begin{theorem}
	\label{thm_small_extensions}
	Assume \autoref{setup_tan_obs} and the notations above.  
	The set of fiber-full extensions of $Z \in \Fib_{\PP_{C}^r/C}$ over $C'$
	$$
	\left\lbrace Z' \in \Fib_{\PP_{C'}^r/C'} \;\text{ such that }\;  Z' \times_{\Spec(C')} \Spec(C) \cong Z \right\rbrace
	$$
	is a pseudotorsor under the action of the subgroup of $\HH^0\big( \PP_\kk^r, \mathcal{N}_{Z_0/\PP_\kk^r} \otimes_\kk \aaa \big)$ given by
	$$
	\left\lbrace \varphi \in \Hom_{\PP_\kk^r}\left(\II_{Z_0}, \OO_{Z_0} \otimes_\kk \aaa\right)  \;\mid\; \HH^i(\PP_\kk^r, \varphi(\nu))=0 \text{ for all $i \ge 1$, $\nu \in \ZZ$} \right\rbrace, 
	$$
	where $\HH^i\left(\PP_\kk^r, \varphi(\nu)\right)$ denotes the natural map $\HH^i\left(\PP_\kk^r, \varphi(\nu)\right) : \HH^i\left(\PP_\kk^r, \II_{Z_0}(\nu)\right) \rightarrow \HH^i\left(\PP_\kk^r, \left(\OO_{Z_0} \otimes_\kk \aaa\right)(\nu)\right)$ induced in cohomology.
\end{theorem}
\begin{proof}
	The proof follows along the same lines of the proof of \autoref{thm_small_extensions_loc}, but now one replaces \autoref{lem_fib_full_crit_loc} by \autoref{lem_fib_full_crit}.
\end{proof}

We  determine sufficient conditions for the existence of obstructions on $\fFib_{Z_0}$. 
For the local fiber-full functor $\fFib_{Z_0} : (\Art_\kk) \rightarrow (\text{Sets})$, we only obtain a partial answer for part (1) of \autoref{rem_tan_obs}.
For a coherent sheaf $\FF$ on $\PP_\kk^r$,  let $\HH_*^i(\FF)$ be the graded $R$-module given by $\HH_*^i(\FF) := \bigoplus_{\nu \in \ZZ} \HH^i(\PP_\kk^r, \FF(\nu))$.

\begin{theorem}
	\label{thm_obs_fib_full}
	Assume \autoref{setup_tan_obs}.
	Let $Z_0 \subset \PP_{\kk}^r$ be a closed subscheme. 
	Let $0 \rightarrow \aaa \rightarrow C' \rightarrow C \rightarrow 0$ be a small extension in $\Art_\kk$ and $Z \in \fFib_{Z_0}(C)$.
	There are elements $\ob(I, C') \in \Ext_{\PP_{\kk}^r}^1(\II_{Z_0}, \OO_{Z_0} \otimes_{\kk} \aaa)$ and $\ob_i(I, C') \in \big[\Ext_R^2\left(\HH_*^i(\OO_{Z_0}), \HH_*^i(\OO_{Z_0}) \otimes_{\kk} \aaa\right)\big]_0$ for all $0 \le i \le r$ such that, if 
	$$
	\ob(I, C') \neq 0 \quad \text{ or } \quad \ob_i(I, C') \neq 0 \text{ for some } 0 \le i \le r,
	$$
	then there is no lifting $Z' \in \fFib_{Z_0}(C')$ of $Z$.
\end{theorem}
\begin{proof}
	The proof is verbatim the one of \autoref{thm_obs_loc}.
\end{proof}

\begin{remark}
	As in \autoref{remark_quest_obs_loc}, we expect that \autoref{thm_obs_fib_full} gives an obstruction space for $\fFib_{Z_0}$.
\end{remark}

Finally, we deal with the pro-representability of the local fiber-full functor $\fFib_{Z_0} : (\Art_\kk) \rightarrow (\text{Sets})$.

\begin{remark}
	The pro-representability of the local functor $\fFib_{Z_0}$ follows directly from the representability of the global functor $\fFib_{\PP_{\kk}^r/\kk}$ (see, e.g., \cite[Proposition 23.3]{hartshorne2010deformation}).
	Alternatively, one may check directly Schlessinger's criterion (see \cite[Theorem 16.4]{hartshorne2010deformation}), where  $(H_4)$ is already granted to us by \autoref{thm_small_extensions}.
\end{remark}

\section{One-parameter flat families with constant cohomology}
\label{sect_one_param}

In this section, we study one-parameter flat families, and we determine an interesting class of families that keep the cohomologies constant. 
These results seem of particular importance for the case of Gr\"obner degenerations. We begin by fixing the following setup for the rest of the section.

\begin{setup}
	\label{setup_one_param}
	Let $\kk$ be a field, $R$ be a standard graded polynomial ring $R = \kk[x_0,\ldots,x_r]$ and $\mm = (x_0,\ldots,x_r) \subset R$ be the graded irrelevant ideal.
	Let $A = \kk[t]$ be a polynomial ring, and $B = A_{(t)}$ be the local ring at the prime ideal $(t) \subset A$.
	Let $S = A[x_0,\ldots,x_r] \cong A \otimes_{\kk} R$ be a polynomial ring with grading induced by $R$.
	For each $q \ge 1$, set $B_q := A/t^qA$ and $S_q := S / t^qS \cong B_q \otimes_{\kk} R$.
\end{setup}

We say that \emph{$\mathcal{Z} \subset \PP_\kk^r \times_\kk \bbA^1_\kk \cong \Proj(S)$ is a one-parameter flat family of closed subschemes with special fiber $Z_0 \subset \PP_{\kk}^r$ and general fiber $Z_1 \subset \PP_{\kk}^r$} if the following two conditions are satisfied: 
\begin{enumerate}[(1)]
	\item $\mathcal{Z} \subset \PP_\kk^r \times_\kk \bbA^1_\kk$ is a closed subscheme flat over $\bbA^1_\kk$.
	\item $Z_0 \cong \mathcal{Z} \times_{\Spec(A)} \Spec(A/tA)$ and $Z_1 \times_\kk \Spec(\kk[t,t^{-1}]) \cong \mathcal{Z} \times_{\Spec(A)} \Spec(\kk[t, t^{-1}])$.
\end{enumerate}
In particular, this implies that for any $0 \neq \alpha\in \kk$, the fiber $\mathcal{Z}_\alpha := \mathcal{Z} \times_{\Spec(A)} \Spec(\kk[t]/(t-\alpha)) \subset \PP_{\kk}^r$ coincides with $Z_1 \subset \PP_\kk^r$.
We say that the family \emph{$\mathcal{Z} \subset \PP_\kk^r \times \bbA_\kk^1$ has constant cohomology} if the equality
$$
\dim_\kk\left(\HH^i(\PP_\kk^r, \OO_{Z_0}(\nu))\right) \, = \, \dim_\kk\left(\HH^i(\PP_\kk^r, \OO_{Z_1}(\nu))\right)
$$
holds for all $i \ge 0$ and $\nu \in \ZZ$.

In more algebraic terms, we say that \emph{$\mathcal{I} \subset S$ is a one-parameter flat family of homogeneous ideals with special fiber $I_0 \subset R$ and general fiber $I_1 \subset R$} if the following two conditions are satisfied: 
\begin{enumerate}[(1)]
	\item $\mathcal{I} \subset S$ is a homogeneous ideal such that $S/\mathcal{I}$ is flat over $A$.
	\item $R/I_0 \cong S/\mathcal{I} \otimes_{A} A/tA$ and $R/I_1 \otimes_\kk \kk[t,t^{-1}] \cong S/\mathcal{I} \otimes_{A} \kk[t, t^{-1}]$.
\end{enumerate}
This implies that for any $0 \neq \alpha\in \kk$, the fiber $\mathcal{I}_\alpha := \mathcal{I} \otimes_{A} \kk[t]/(t-\alpha) \subset R$ coincides with $I_1 \subset R$.
We say that the family \emph{$\mathcal{I} \subset S$ has constant cohomology} if the equality
$$
\dim_\kk\left(\left[\HL^i(R/I_0)\right]_\nu\right) \, = \, \dim_\kk\left(\left[\HL^i(R/I_1)\right]_\nu\right)
$$
holds for all $i \ge 0$ and $\nu \in \ZZ$.

The prime example of one-parameter flat families is that of \emph{Gr\"obner degenerations} (see \cite[\S 15.7]{EISEN_COMM}).
If one has a monomial order $>$ on $R$ and a homogeneous ideal $I \subset R$, then one can always find a one-parameter flat family $\mathcal{I} \subset S$ such that the initial ideal $\iniTerm_>(I) \subset R$ is the special fiber and $I \subset R$ is the general fiber (see \cite[Theorem 15.17]{EISEN_COMM}, \cite[Proposition 3.2.4]{MONOMIAL_IDEALS}, \cite[Proposition 3.5]{CONFERENCE_LEVICO}).

The following proposition deals with the flat extensions to $S_{q+1}$ of a homogeneous ideal in $S_q$ that is fiber-full over $B_q$, and characterizes when the extension is fiber-full over $B_{q+1}$.

\begin{proposition}
	\label{prop_one_param}
	Let $J_q \subset S_q$ be a homogeneous ideal such that $S_q/J_q$ is fiber-full over $B_q$ and $S_q/J_q \otimes_{B_q} B_q/tB_q \cong R/I_0$.
	Then, for any homogeneous ideal $J_{q+1}$ such that $S_{q+1}/J_{q+1}$ is $B_{q+1}$-flat and $S_{q+1}/J_{q+1} \otimes_{B_{q+1}} B_q \cong S_q / J_q$, we have the following induced commutative diagram of $R$-linear maps homogeneous of degree zero
	 		\begin{equation*}		
	 	\begin{tikzpicture}[baseline=(current  bounding  box.center)]
	 		\matrix (m) [matrix of math nodes,row sep=3.5em,column sep=3em,minimum width=2em, text height=1.5ex, text depth=0.25ex]
	 		{
	 			0 & J_q \otimes_{B_q} (t) B_{q+1} & J_{q+1} &  I_0 & 0  \\
	 			0 & J_q \otimes_{B_q} (t)B_{q+1} & S_q \otimes_{B_q} (t)B_{q+1} &  S_q/J_q \otimes_{B_q} (t)B_{q+1} & 0,  \\
	 		};						
	 		\path[-stealth]
	 		(m-1-1) edge  (m-1-2)
	 		(m-1-2) edge  (m-1-3)
	 		(m-1-3) edge  (m-1-4)
	 		(m-1-4) edge  (m-1-5)
	 		(m-2-1) edge  (m-2-2)
	 		(m-2-2) edge (m-2-3)
	 		(m-2-3) edge  (m-2-4)
	 		(m-2-4) edge  (m-2-5)
	 		(m-1-2) edge node [right] {$=$} (m-2-2)
	 		(m-1-3) edge node [right] {$\pi$} (m-2-3)
	 		(m-1-4) edge node [right] {$\Phi_{J_{q+1}}$} (m-2-4)
	 		;
	 	\end{tikzpicture}	
	 \end{equation*}	
 	and that $S_{q+1}/J_{q+1}$ is fiber-full over $B_{q+1}$ if and only if $\HL^i(\Phi_{J_{q+1}}) = 0$ for all $1 \le i \le r$.
\end{proposition}
\begin{proof}
	Let $\pi : S_{q+1} \rightarrow S_{q} \otimes_{B_q} (t)B_{q+1}$ be the natural $R$-linear map given by 
	$$
	f_0 + f_1t + \cdots + f_qt^q \in S_{q+1} \; \mapsto\; f_1t + \cdots + f_qt^q \in S_q \otimes_{B_q} (t)B_{q+1}
	$$		
	where $f_i \in R$.
	Note that this map is homogeneous of degree zero.
	By an abuse of notation, denote also by $\pi$ the restriction of this map to $J_{q+1}$.
	From the assumptions, we obtain the short exact sequence $0 \rightarrow J_q \otimes_{B_q} (t) B_{q+1} \rightarrow J_{q+1} \rightarrow I_0 \rightarrow 0$.
	This short exact sequence and the map $\pi$ give the following natural commutative diagram 
		 		\begin{equation*}		
		\begin{tikzpicture}[baseline=(current  bounding  box.center)]
			\matrix (m) [matrix of math nodes,row sep=3.5em,column sep=3em,minimum width=2em, text height=1.5ex, text depth=0.25ex]
			{
				0 & J_q \otimes_{B_q} (t) B_{q+1} & J_{q+1} &  I_0 & 0  \\
				0 & J_q \otimes_{B_q} (t)B_{q+1} & S_q \otimes_{B_q} (t)B_{q+1} &  S_q/J_q \otimes_{B_q} (t)B_{q+1} & 0,  \\
			};						
			\path[-stealth]
			(m-1-1) edge  (m-1-2)
			(m-1-2) edge  (m-1-3)
			(m-1-3) edge  (m-1-4)
			(m-1-4) edge  (m-1-5)
			(m-2-1) edge  (m-2-2)
			(m-2-2) edge (m-2-3)
			(m-2-3) edge  (m-2-4)
			(m-2-4) edge  (m-2-5)
			(m-1-2) edge node [right] {$=$} (m-2-2)
			(m-1-3) edge node [right] {$\pi$} (m-2-3)
			;
			\draw[->,dashed] (m-1-4)--(m-2-4);
		\end{tikzpicture}	
	\end{equation*}	
	and so we get a unique $R$-linear map $\Phi_{J_{q+1}} : I_0 \rightarrow S_q/J_q \otimes_{B_q} (t)B_{q+1}$ making the diagram commute.
	
	By making similar computations to the ones made in \autoref{thm_tan_space_loc} or \autoref{thm_small_extensions_loc}, it follows that $\HL^i(J_{q+1}) \twoheadrightarrow \HL^i(I_0)$ if and only if $\HL^i(\Phi_{J_{q+1}}) = 0$.
	Finally, the result follows from the local criterion for fiber-fullness given in \autoref{lem_fib_full_crit_loc}.
\end{proof}

The main result of this section is the following sufficient condition for a one-parameter flat family of homogeneous ideals to have constant cohomology.

\begin{theorem}
	\label{thm_ one_param_ideals}
	Assume \autoref{setup_one_param}.
	Let $\mathcal{I} \subset S$ be a one-parameter flat family of homogeneous ideals with special fiber $I_0 \subset R$ and general fiber $I_1 \subset R$. 
	If the zero map is the only zero-degree homogeneous $R$-linear map from $\HL^{i-1}(R/I_0)$ to $\HL^{i}(R/I_0)$ for all $1 \le i \le r$, i.e.,
	$$
	\left[\Hom_R\Big(\HL^{i-1}(R/I_0),\, \HL^{i}(R/I_0)\Big)\right]_0 = 0 \quad \text{for all $1 \le i \le r$},
	$$
	then the family $\mathcal{I}$ has constant cohomology.
\end{theorem}
\begin{proof}
	First, note that we get $\left[\Hom_R\left(\HL^{i}(I_0),\, \HL^{i}(R/I_0)\right)\right]_0 = 0$ for all $1 \le i \le r$, because there is an isomorphism $\HL^{i-1}(R/I_0) \cong \HL^i(I_0)$ for all $1 \le i \le r$.
	
	We prove by induction that $J_q := \II \otimes_{A} B_q \subset S_q$ is a homogeneous ideal such that 
	$S_q/J_q$ is fiber-full over $B_q$ for all $q \ge 1$.
	The case $q = 1$ is clear.
	Suppose that $S_q/J_q$ is fiber-full over $B_q$.
	Thus, \autoref{prop_one_param} gives an $R$-linear map $\Phi_{J_{q}} : I_0 \rightarrow S_{q-1}/J_{q-1} \otimes_{B_{q-1}} (t)B_{q} \cong S_{q-1}/J_{q-1}$ such that $\HL^i(\Phi_{J_{q}}) = 0$ for all $1 \le i \le r$.
	On the other hand, since $S_{q+1}/J_{q+1}$ is $B_{q+1}$-flat, due to \autoref{prop_one_param}, there is an $R$-linear map $\Phi_{J_{q+1}} : I_0 \rightarrow S_q/J_q \otimes_{B_q} (t)B_{q+1} \cong S_q/J_q$ such that $S_{q+1}/J_{q+1}$ is fiber-full over $B_{q+1}$ if and only if $\HL^i(\Phi_{J_{q+1}}) = 0$ for all $1 \le i \le r$.
	Note that by construction the $R$-linear maps $\Phi_{J_{q}}$  and $\Phi_{J_{q+1}}$ are $R$-linear homogeneous of degree zero, and that they fit into the following commutative diagram
		 		\begin{equation*}		
		\begin{tikzpicture}[baseline=(current  bounding  box.center)]
			\matrix (m) [matrix of math nodes,row sep=3.5em,column sep=3.7em,minimum width=2em, text height=1.5ex, text depth=0.25ex]
			{
				 &  & I_0 &  &   \\
				0 & R/I_0 \otimes_{\kk} (t^{q-1})B_{q} & S_q/J_q  &  S_{q-1}/J_{q-1} & 0.  \\
			};						
			\path[-stealth]
			(m-2-1) edge  (m-2-2)
			(m-2-2) edge (m-2-3)
			(m-2-3) edge  (m-2-4)
			(m-2-4) edge  (m-2-5)
			(m-1-3) edge node [left] {$\Phi_{J_{q+1}}$} (m-2-3)
			(m-1-3) edge node [above] {$\Phi_{J_{q}}$} (m-2-4)
			;
		\end{tikzpicture}	
	\end{equation*}	
	Since by assumption $S_q/J_q$ is fiber-full over $B_q$, the natural map $\HL^i(S_q/J_q) \rightarrow \HL^i(S_{q-1}/J_{q-1})$ is surjective for all $i \ge 0$, and so we obtain the following commutative diagram
	\begin{equation*}		
		\begin{tikzpicture}[baseline=(current  bounding  box.center)]
			\matrix (m) [matrix of math nodes,row sep=4.5em,column sep=3.7em,minimum width=2em, text height=1.5ex, text depth=0.25ex]
			{
				&  & \HL^i(I_0) &  &   \\
				0 & \HL^i(R/I_0 \otimes_{\kk} (t^{q-1})B_{q}) & \HL^i(S_q/J_q)  &  \HL^i(S_{q-1}/J_{q-1}) & 0.  \\
			};						
			\path[-stealth]
			(m-2-1) edge  (m-2-2)
			(m-2-2) edge (m-2-3)
			(m-2-3) edge  (m-2-4)
			(m-2-4) edge  (m-2-5)
			(m-1-3) edge node [left] {$\HL^i(\Phi_{J_{q+1}})$} (m-2-3)
			(m-1-3) edge node [above] {$\qquad\;\HL^i(\Phi_{J_{q}})$} (m-2-4)
			;
		\end{tikzpicture}	
	\end{equation*}	
	As we already know that $\HL^i(\Phi_{J_{q}}) = 0$, it follows that the map $\HL^i(\Phi_{J_{q+1}})$ factors as a map 
	$$
	\HL^i(\Phi_{J_{q+1}}) \,:\, \HL^i(I_0) \longrightarrow  \HL^i(R/I_0 \otimes_{\kk} (t^{q-1})B_{q}) \cong \HL^i(R/I_0).
	$$
	This latter map is $R$-linear homogeneous of degree zero from $\HL^i(I_0)$ to $\HL^i(R/I_0)$, and so it holds that $\HL^i(\Phi_{J_{q+1}}) = 0$.
	Thus, \autoref{prop_one_param} implies that 
	$S_{q+1}/J_{q+1}$ is fiber-full over $B_{q+1}$.
	So, the induction step is complete.
	
	Since we have shown that $S_q/J_q$ is fiber-full over $B_q$ for all $q \ge 1$, the local criterion for fiber-fullness (see \autoref{thm_fib_full_mod}) implies that $S/\II \otimes_{A} B$ is fiber-full over $B$.
	Due to the flatness of cohomologies and the arbitrary base change under the fiber-full condition (see \autoref{lem_base_ch_fib_full}), we obtain the equalities
	\begin{align*}
		\dim_\kk\left(\left[\HL^i(R/I_0)\right]_\nu\right) &= \dim_\kk\left(\left[\HL^i\left(S/\II \otimes_{A} B\right)\right]_\nu \otimes_B B/tB \right) \\
		&= \dim_{\kk(t)}\left(\left[\HL^i\left(S/\II \otimes_{A} B\right)\right]_\nu \otimes_B \kk(t) \right) \\
		&= \dim_\kk\left(\left[\HL^i(R/I_1)\right]_\nu\right) 
	\end{align*}
	for all $i \ge 0$ and $\nu \in \ZZ$.
	This finishes the proof of the theorem.
\end{proof}

The following corollary provides an effective method to discover flat families with constant cohomology.

\begin{corollary}
	\label{cor_criterion_Ext}
	Assume \autoref{setup_one_param}.
	Let $\mathcal{I} \subset S$ be a one-parameter flat family of homogeneous ideals with special fiber $I_0 \subset R$ and general fiber $I_1 \subset R$. 
	If the zero map is the only zero-degree homogeneous $R$-linear map from $\Ext_R^{i-1}(R/I_0, R)$ to $\Ext_R^{i}(R/I_0, R)$ for all $1 \le i \le r$, i.e.,
	$$
	\left[\Hom_R\Big(\Ext_R^{i-1}(R/I_0, R),\, \Ext_R^{i}(R/I_0, R)\Big)\right]_0 = 0 \quad \text{for all $1 \le i \le r$},
	$$
	then the family $\mathcal{I}$ has constant cohomology.
\end{corollary}
\begin{proof}
	This is a consequence of \autoref{thm_ one_param_ideals} and the graded local duality theorem (see, e.g., \cite[\S 3.6]{BRUNS_HERZOG}).
\end{proof}

Next, we enunciate our main result for the case one-parameter flat families of closed subschemes.
For a coherent sheaf $\FF$ on $\PP_\kk^r$, let  $\HH_*^i(\FF)$ denote the graded $R$-module given by $\HH_*^i(\FF) := \bigoplus_{\nu \in \ZZ} \HH^i(\PP_\kk^r, \FF(\nu))$.
 
\begin{theorem}
	\label{thm_const_one_param_sch}
	Assume \autoref{setup_one_param}.
	Let $\mathcal{Z} \subset \PP_\kk^r \times_\kk \bbA^1_\kk \cong \Proj(S)$ be a one-parameter flat family of closed subschemes with special fiber $Z_0 \subset \PP_{\kk}^r$ and general fiber $Z_1 \subset \PP_{\kk}^r$. 
	If the zero map is the only zero-degree homogeneous $R$-linear map from $\HH_*^{i-1}(\OO_{Z_0})$ to $\HH_*^{i}(\OO_{Z_0})$ for all $1 \le i \le r-1$, i.e.,
	$$
	\left[\Hom_R\Big(\HH_*^{i-1}(\OO_{Z_0}),\, \HH_*^{i}(\OO_{Z_0})\Big)\right]_0 = 0 \quad \text{for all $1 \le i \le r-1$},
	$$
	then the family $\mathcal{Z}$ has constant cohomology.
\end{theorem}
\begin{proof}
	The proof is obtained along the same lines of the one of \autoref{thm_ one_param_ideals}.
\end{proof}

\section{A local comparison of the fiber-full scheme and the Hilbert scheme}
\label{sect_comparison_local_functors}

The goal of this section is to compare local deformation functors of $\Fib_{R/\kk}$ and $\HS_{R/\kk}$.
The results of this section are inspired by our study of one-parameter flat families in \autoref{sect_one_param}. 
The setup below is used throughout this section. 

\begin{setup}
	\label{setup_comparison}
	Let $\kk$ be a field, $R = \kk[x_0,\ldots,x_r]$ be a standard graded polynomial ring, and $\mm = (x_0,\ldots, x_r) \subset R$ be the graded irrelevant ideal.
\end{setup}

Let $I_0 \subset R$ be a homogeneous ideal.
We consider the deformation functors $\fFib_{I_0} : (\Art_\kk) \rightarrow (\text{Sets})$ and $\fHS_{I_0} : (\Art_\kk) \rightarrow (\text{Sets})$ given by 
$$
\fFib_{I_0}(A) \,=\, \big\{ I \in \fFib_{R/\kk}(A)  \mid I \otimes_{A} \kk \cong I_0\big\} \quad \text{ and } \quad \fHS_{I_0}(A) \,=\, \big\{ I \in \fHS_{R/\kk}(A)  \mid I \otimes_{A} \kk \cong I_0\big\},
$$
respectively.
By construction $\fFib_{I_0}$ is a subfunctor of $\fHS_{I_0}$.
The following theorem gives a family of ideals for which $\fFib_{I_0}$ and $\fHS_{I_0}$ coincide.

\begin{theorem}
	Assume \autoref{setup_comparison}.
	Let $I_0 \subset R$ be a homogeneous ideal such that 
	$$
	\left[\Hom_R\Big(\HL^{i-1}(R/I_0),\, \HL^{i}(R/I_0)\Big)\right]_0 = 0 \quad \text{for all $1 \le i \le r$}.
	$$
	Then the local deformation functors $\fFib_{I_0}$ and $\fHS_{I_0}$ coincide.
\end{theorem}
\begin{proof}
		The case $I_0 = 0$ is clear. 
		Thus we assume that $I_0 \neq 0$, and so $\HL^{r+1}(R/I_0) = 0$.
		Therefore, we can actually assume that $\left[\Hom_R\Big(\HL^{i-1}(R/I_0),\, \HL^{i}(R/I_0)\Big)\right]_0 = 0$ for all $i \ge 1$.
	
		We need to prove that $\fFib_{I_0}(A) = \fHS_{I_0}(A)$ for any $A \in \Art_\kk$.
		We proceed by induction on $\text{length}(A)$.
		The base case $\text{length}(A)=1$ is clear because we necessarily have $A = \kk$.
		Thus, we assume that $\text{length}(A) > 1$.
		We may choose an ideal $\aaa \subset A$ generated by a socle element of $A$ to obtain a short exact sequence 
	$$
	0 \rightarrow \aaa \rightarrow A \rightarrow A' \rightarrow 0
	$$	
	with $\aaa \cong \kk$ and $\text{length}(A') = \text{length}(A)-1$.
	Let $S = A[x_0,\ldots,x_r]$ and $S' = A'[x_0,\ldots,x_r]$.
	Choose $I \in \fHS_{I_0}(A)$ and set $I' = I \otimes_{A} A' \in \in \fHS_{I_0}(A')$.
	By tensoring with $- \otimes_A S/I$, since $S/I$ is flat over $A$, we obtain the short exact sequence 
	$$
	0 \rightarrow R/I_0 \rightarrow S/I \rightarrow S'/I' \rightarrow 0.
	$$
	Then we get an induced long exact sequence in cohomology
	$$
	\HL^{i-1}(S'/I') \,\xrightarrow{\delta_{i-1}}\, \HL^i(R/I_0) \,\rightarrow\, \HL^i(S/I) \,\rightarrow\, \HL^{i}(S'/I') \,\xrightarrow{\delta_i}\, \HL^{i+1}(R/I_0).
	$$
	The inductive hypothesis and \autoref{lem_base_ch_fib_full} yield the isomorphism $\HL^{i}(S'/I') \otimes_{A'} \kk \xrightarrow{\cong} \HL^i(R/I_0)$.
	Then the map $\delta_i$ gets the following factorization 
	$$
	\HL^{i}(S'/I') \; \twoheadrightarrow \; \HL^{i}(S'/I') \otimes_{A'} \kk \cong \HL^i(R/I_0) \; \xrightarrow{\beta_i} \; \HL^{i+1}(R/I_0)
	$$
	where $\beta_i$ is an $R$-linear map of degree $0$.
	The condition $\left[\Hom_R\left(\HL^i(R/I_0),\, \HL^{i+1}(R/I_0)\right)\right]_0 = 0$ implies that $\beta_i = 0$, hence $\delta_i = 0$.
	We have proved that the natural map $\HL^i(S/I)\rightarrow \HL^i(S'/I')$ is surjective, and the natural map $\HL^i(S'/I') \rightarrow \HL^i(R/I_0)$ is surjective by induction. 
	Therefore, it follows that the natural map 
	$$
	\HL^i(S/I) \,\rightarrow\, \HL^i(R/I_0)
	$$
	is also surjective for all $i \ge 0$.
	Finally, by applying \autoref{thm_fib_full_mod}, we obtain that $S/I$ is fiber-full over $S$.
	This shows that $I \in \fFib_{I_0}(A)$, as required.
\end{proof}

\section{Examples}
\label{sect_examp}

In this section, we give several examples to illustrate the main results in this paper.
We start by providing a couple of examples that indicate the broad applicability of \autoref{thm_ one_param_ideals}.
We are particularly interested in examples of non-squarefree monomial ideals that satisfy the criterion given in \autoref{thm_ one_param_ideals}. 
These families of ideals are a further advance in the context of Gr\"obner degenerations (that preserve cohomology), as they are not covered by the result of Conca-Varbaro \cite[Theorem 1.3]{CONCA_VARBARO}.
By \cite[Proposition 2.3]{CONCA_VARBARO},
squarefree monomial ideals are a particular class of ideals whose quotients give a \emph{cohomologically full ring} (\cite{COHOM_FULL_RINGS, KOLLAR_KOVACS}). 
It should be mentioned that the same aforementioned result of Conca-Varbaro also holds for homogeneous ideals whose quotient is a cohomologically full ring (see also \cite[Remark 2.9]{CONCA_VARBARO}).

\begin{remark}
	Assume \autoref{setup_one_param}.
	Let $\mathcal{I} \subset S$ be a one-parameter flat family of homogeneous ideals with special fiber $I_0 \subset R$ and general fiber $I_1 \subset R$. 
	If $R/I_0$ is cohomologically full, then   one can show that $\II$ has constant cohomology by using  \autoref{thm_fib_full_mod}.
\end{remark}

We begin with an example on residual intersections.

\begin{example}[Residual intersections]
	\label{examp_residual}
	Consider the $2 \times n$ generic matrix  
	$$
	M = \begin{pmatrix} x_{11} & x_{12} & \cdots & x_{1n} \\ x_{21} & x_{22} & \cdots & x_{2n}\end{pmatrix} 
	$$
	over $R = \kk[x_{11},\dots,x_{2n}]$ and let $I= I_2(M)$. Let $g := n-1$ denote the codimension and $\ell = 2n-3$ be the analytic spread. Let $a_1,\dots,a_{\ell}$ be quadrics in $I$ that generate a minimal reduction. Let $J_i = (a_1,\dots,a_i) \subseteq I$ and assume that $J_i:I$ is a geometric $i$-residual intersection $i \leq \ell - 1$. For $n \leq  i \leq \ell -1$, the ideal $J_i:I$ is not Cohen-Macaulay and \cite[Theorem 1.2]{EISENBUD_ULRICH} shows that 
	$$
	\HL^j(R/(J_i:I)) \ne 0 \text{ if and only if } j \in \{2n-i-3,2n-i\}.
	$$ 
Thus the hypothesis of \autoref{thm_ one_param_ideals} is satisfied and it follows that, for each $n \leq  i \leq \ell -1$, any one-parameter degeneration to $J_i:I$ has constant cohomology.
\end{example}


The following example gives two simple ideals whose quotients are not cohomologically full rings, but do satisfy the criterion of \autoref{thm_ one_param_ideals}. 

\begin{example} 
	\label{examp_first}
	We will show that the cohomology is constant in one-parameter flat families that have the following ideals $R = \kk[x,y,z]$ as the special fiber
$$
I_1 = (x^2,y^2,xy,xz,yz)  \,\, \text{ and } \,\,I_2 = (x^2y, xz^2, y^2z, z^3).
$$
Both ideals $I_1$ and $I_2$ are $1$-dimensional and not Cohen-Macaulay, and so both $R/I_1$ and $R/I_2$ are not cohomologically full rings (see \cite[Corollary 2.6]{COHOM_FULL_RINGS}).
Consider the \v{C}ech complex 
$$
\small
C_\mm^\bullet : \quad 0 \to R \xrightarrow{\begin{pmatrix}1 \\ 1 \\ 1\end{pmatrix}} 
	R_x \oplus R_y \oplus R_z \xrightarrow{\begin{pmatrix}-1 & 1 & 0 \\ -1 & 0 & 1 \\ 0 & -1 & 1 \end{pmatrix}} 
	R_{xy} \oplus R_{xz} \oplus R_{yz} \xrightarrow{\begin{pmatrix}1  \\ -1  \\ 1  \end{pmatrix}^{T}} 
	R_{xyz} \to 0.
$$
\begin{enumerate}[(1)]
\item Let $\overline{R} = R/I_1$ and notice that the \v{C}ech complex $C_\mm^\bullet \otimes_R \overline{R}$ collapses to
$$
0 \to \overline{R} \to \overline{R}_{z} \to 0.
$$
It follows that the non-zero cohomology groups are $\HH^0_{\mm}(\overline{R}) = (x,y)\overline{R} \cong \left(\kk(-1)\right)^2$ and
	$\HH^1_{\mm}(\overline{R}) = \overline{R}_z/\overline{R} \cong \bigoplus_{i=1}^\infty \kk(i)$. 
	Thus $\left[ \Hom_{R}(\HH^0_{\mm}(\overline{R}),\HH^1_{\mm}(\overline{R}))\right]_0 = 0$ and by \autoref{thm_ one_param_ideals}  cohomology is preserved under flat degenerations to the ideal $I_1 \subset R$.
	\item After setting $\overline{R} = R/I_2$,  the \v{C}ech complex $C_\mm^\bullet$ collapses to
$$
0 \to \overline{R} \xrightarrow{\varphi = \begin{pmatrix} 1 & 1 \end{pmatrix}^T} \overline{R}_{x} \oplus \overline{R}_{y} \to 0.
$$
The non-zero cohomology groups are in positions $0$ and $1$, with 
$
\HH^0_{\mm}(\overline{R})=(z^2,yz)\overline{R} \cong \left(\kk(-2)\right)^2 \oplus \left(\kk(-3)\right)^2 
$. 
Moreover, one can check that $\mm^2\left( \overline{R}_x \oplus \overline{R}_y \right)$ lies in the image of $\varphi$. It follows that $\HH^1_{\mm}(\overline{R})$ is supported in degrees $\leq 1$ and thus $\left[ \Hom_{R}(\HH^0_{\mm}(\overline{R}),\HH^1_{\mm}(\overline{R}))\right]_0 = 0$ as required.
Due to \autoref{thm_ one_param_ideals}  cohomology is preserved under flat degenerations to the ideal $I_2 \subset R$.
\end{enumerate}
\end{example}

We now test the applicability of \autoref{cor_criterion_Ext} with the following simple example of a non-squarefree monomial ideal.

\begin{example} 
	\label{examp_second}
	Consider the ideal $I = (xy^2, xz) \subseteq R = \kk[x,y,z]$. 
	By using the Taylor resolution, we can compute that $\Ext_R^1(R/I, R) \cong \frac{R}{(x)}(1)$, $\Ext_R^2(R/I, R)  \cong \frac{R}{(y^2,z)}(4)$ and that all the other modules $\Ext_R^i(R/I,R)$  vanish.
	In this case, we even have that $\Hom_R\left(\frac{R}{(x)}, \frac{R}{(y^2,z)}\right) = 0$.
	Therefore, \autoref{cor_criterion_Ext} implies that cohomology is preserved under flat degenerations to the ideal $I \subset R$.
	Notice that $\left[\HL^1(R/I)\right]_\nu \neq 0$ for $\nu \le 1$ and $\left[\HL^2(R/I)\right]_\nu \neq 0$ for $\nu \le -2$, but we have $\Hom_{R}\left(\HL^1(R/I), \HL^2(R/I)\right) = 0$.
\end{example}

\begin{example}
Here are two more examples for which degenerations to them preserve cohomology: $I_1 = (x^3,xy,xz+y^2) \subseteq \kk[x,y,z]$ and $I_2 = (x_1^2,x_1x_2,x_1x_4,x_2^2,x_3^2,x_3x_4) \subseteq \kk[x_1,\dots,x_4]$. The verification is analogous to the previous examples.
\end{example}

\begin{example} 
	Let $I \subseteq R = \kk[x_1,\ldots,x_r]$ be a squarefree monomial ideal such that $\dim (R/I) \leq 2$. 
	Then the Stanley-Reisner complex, $\Delta$, associated to $I$ is at most one dimensional. 
	Thus for any subcomplex $H \subseteq \Delta$ we have the reduced cohomology $\widetilde{\HH}^j(H;\kk)= 0$ for all $j\geq 1$. We want to show that there are no non-zero maps $ \left[\HH^{i-1}_{\mm}(R/I)\right]_{\mathbf{b}} \to \left[\HH^i_{\mm}(R/I)\right]_{\mathbf{b}}$ for any $i\geq 1$ and $\mathbf{b} \in \ZZ^r$. 
	By Hochster's formula \cite[Theorem 13.13]{MILLER_STURMFELS} this correspond to a map
$$
\widetilde{\HH}^{i-\lvert \sigma \rvert-2}(\text{lk}_{\Delta}\sigma;\kk) \longrightarrow \widetilde{\HH}^{i-\lvert \sigma \rvert-1}(\text{lk}_{\Delta}\sigma;\kk) 
$$
where $\sigma = \text{supp}(\mathbf{b})$. However, as we just mentioned, since $\Delta$ has at most dimension $1$, either the domain or codomain of the map is $0$. Thus the hypothesis of \autoref{thm_ one_param_ideals} is satisfied and we obtain the constancy of cohomology for flat degenerations to $I \subset R$. 
This is a simple illustrative instance of the result of Conca-Varbaro \cite[Theorem 1.3]{CONCA_VARBARO}.
\end{example}

It is also possible that  \autoref{thm_ one_param_ideals} fails while the ideal is monomial squarefree. 

\begin{example} 
	\label{examp_sqrfree_bad}
	Consider the squarefree monomial ideal in $R = \kk[x_1\dots x_6]$ given by 
$$
I = (x_1x_4,x_2x_4,x_3x_4,x_1x_5,x_2x_5,x_3x_5,x_1x_6,x_2x_6,x_3x_6,x_1x_2x_3).
$$
This is the Stanley-Reisner ideal associated the simplicial complex $\Delta$ which is a disjoint union of a 3-cycle and a 2-simplex:
$$
\begin{tikzpicture}
  [dot/.style={draw,fill,circle,inner sep=1pt}]
  \node[dot] (n2) at (4,1)  {};
  \node[dot] (n3) at (3,0)  {};
  \node[dot] (n4) at (5,0) {};
  \node[dot] (n5) at (7,1)  {};
  \node[dot] (n6) at (6,0)  {};
  \node[dot] (n7) at (8,0)  {};

  \foreach \from/\to in {n2/n3,n2/n4,n3/n4,n5/n6,n5/n7,n5/n6,n6/n7}
    \draw (\from) -- (\to);
    
    \draw [fill=gray, opacity=0.3] (7,1)--(6,0)--(8,0);
\end{tikzpicture}
$$
By \cite[Theorem 1.3]{CONCA_VARBARO}, any one-parameter flat degeneration to $I$ keeps cohomology constant. 
On the other hand, we will show that there is a non-zero homogeneous $R$-module homomorphism of degree zero $\left[\HH^{1}_{\mm}(R/I)\right]_0 \to \left[\HH^{2}_{\mm}(R/I)\right]_0$.  
By using Hochster's formula \cite[Theorem 13.13]{MILLER_STURMFELS},  we obtain 
$$ 
\HH^{1}_{\mm}(R/I) = \left[\HH^{1}_{\mm}(R/I)\right]_{0} = \left[\HH^{1}_{\mm}(R/I)\right]_{(0,\dots,0)}  \cong \widetilde{\HH}^{0}(\text{link}_{\Delta}(\emptyset);\kk) = \kk
$$ 
and 
$$ 
\left[\HH^{2}_{\mm}(R/I)\right]_{\ge 0} = \left[\HH^{2}_{\mm}(R/I)\right]_{(0,\dots,0)}  \cong \widetilde{\HH}^{1}(\text{link}_{\Delta}(\emptyset);\kk) = \kk.
$$ 
Thus the map on $\kk$-vector spaces $\kk \cong \HH^{1}_{\mm}(R/I) \to \left[\HH^{2}_{\mm}(R/I)\right]_0 \cong \kk$ given by $1\mapsto 1$ extends to a homogeneous $R$-module homomorphism of degree zero. 
This invalidates the hypothesis of \autoref{thm_ one_param_ideals}.
\end{example}

In \cite[Section 6]{FIBER_FULL_SCHEME} we showed that $\Hilb^{3m+1}_{\PP_\kk^3}$ decomposes into a union of two fiber-full subschemes, both of which are smooth. 
This decomposition relied on older technical computations of an analytic neighborhood of a specific Borel-fixed point in the full Hilbert scheme. We can verify the smoothness of the fiber-full schemes directly using \autoref{thm_tan_space}. 

\begin{example} \label{TWISTED_CUBIC_EXAMPLE} The Hilbert scheme $\Hilb^{3m+1}_{\PP_\kk^3}$ is a union of two smooth irreducible components such that the general member of $H$ parametrizes a twisted cubic, and the general member of $H'$ parametrizes a plane cubic union an isolated point. 
The component $H$ is $12$-dimensional while $H'$ is $15$-dimensional. 
By \cite[Example 6.2]{FIBER_FULL_SCHEME}, we have the decomposition
$$
\Hilb^{3m+1}_{\PP^3_{\kk}} = \Fib^{(\mathbf{h},0,0)}_{\PP^3_{\kk}} \sqcup \Fib^{(\mathbf{h}',0,0)}_{\PP^3_{\kk}} = (H - H \cap H') \sqcup H',
$$
where $\mathbf{h} = (h_0,h_1), \mathbf{h}' = (h_0',h_1') : \ZZ^2 \rightarrow \NN^2$ are two tuples of functions.
In \cite[Section 4]{TWISTED_CUBIC} the authors show that $\Fib^{\mathbf{h}',0,0}_{\PP^3_{\kk}}$ is smooth away from $H \cap H'$. 
Let $R = \kk[x,y,z,w]$ and set $\PP_\kk^3 = \Proj(R)$.
Furthermore, they show that the every point in $H \cap H'$ admits a flat degeneration to $I = (x^3,xz,yz,z^2) \subset R$ and that the tangent space to $Z = \Proj(R/I) \subset \PP_{\kk}^3$ in the Hilbert scheme is $16$-dimensional. 
Since $\Fib^{\mathbf{h}',0,0}_{\PP^3_{\kk}} = H'$ is $15$-dimensional, to conclude that $Z$ gives a smooth point on its fiber-full scheme, it suffices to show that 
$$
\dim_\kk \left(T_{[Z]} \Fib^{\mathbf{h}',0,0}_{\PP^3_{\kk}}\right) \,<\, \dim\left( \Hilb^{3m+1}_{\PP^3_{\kk}}\right)= 16.
$$
From the Comparison Theorem of  \cite{TWISTED_CUBIC}, we get $T_{[Z]} \Hilb_{\PP^3_{\kk}} = T_{[I]}\HS_R = \left[\Hom_R(I, R/I)\right]_0$.
By \autoref{thm_tan_space}, it is enough to show that there is some $\varphi \in \left[\Hom_R(I, R/I)\right]_0$ for which the induced map
$\HL^2(\varphi):  \HH^2_{\mm}(I) \cong \HH^1_{\mm}(R/I) \to \HH^2_{\mm}(R/I)$ is non-zero.
Let $\overline{R} = R/I$, and note that the associated \v{C}ech complex  $C_\mm^\bullet \otimes_R \overline{R}$ collapses to
$$
0 \to \overline{R} \xrightarrow{{\begin{pmatrix} 1  & 1 \end{pmatrix}}^T}
	\overline{R}_y \oplus \overline{R}_w \xrightarrow{\begin{pmatrix} -1 & 1 \end{pmatrix}} \overline{R}_{yw}\to 0.
$$
From this one can obtain that 
$$
\HH^1_{\mm}(\overline{R}) \cong \bigoplus_{i <0} \kk \cdot zw^{i} \quad \text{ and } \quad
\HH^2_{\mm}(\overline{R}) \cong \bigoplus_{i,j<0} \kk \cdot \left( 1,x,x^2 \right) \cdot y^iw^j. 
$$
Let $\varphi \in \left[\Hom_R(I, R/I)\right]_0$ be the tangent vector corresponding to $\frac{\partial}{\partial u_1}$ from \cite[proof of Lemma 6]{TWISTED_CUBIC}, i.e., $\varphi$ maps $yz \mapsto x^2$ and all other generators of $I$ to $0$. 
The map $\HL^2 (\varphi) : \HL^2(I) \rightarrow \HL^2(\overline{R})$ can be computed from the following map between complexes
	\begin{equation*}		
	\begin{tikzpicture}[baseline=(current  bounding  box.center)]
		\matrix (m) [matrix of math nodes,row sep=2.8em,column sep=1.5em,minimum width=2em, text height=1.5ex, text depth=0.25ex]
		{
			C_\mm^\bullet \otimes_R I: \quad 0 & I & C_\mm^1 \otimes_R I &  C_\mm^2 \otimes_R I & C_\mm^3 \otimes_R I & C_\mm^4 \otimes_R I & 0\\
			C_\mm^\bullet \otimes_R \overline{R}: \quad 0 & \overline{R} & \overline{R}_y \oplus \overline{R}_w &  \overline{R}_{yw} & 0  \\
		};						
		\path[-stealth]
		(m-1-1) edge (m-1-2)
		(m-1-2) edge (m-1-3)
		(m-1-3) edge (m-1-4)
		(m-1-4) edge (m-1-5)
		(m-1-5) edge (m-1-6)
		(m-1-6) edge (m-1-7)
		(m-2-1) edge (m-2-2)
		(m-2-2) edge (m-2-3)
		(m-2-3) edge (m-2-4)
		(m-2-4) edge (m-2-5)
		(m-1-2) edge (m-2-2)
		(m-1-3) edge (m-2-3)
		(m-1-4) edge (m-2-4)
		(m-1-5) edge (m-2-5)
		;
		\end{tikzpicture}	
	\end{equation*}	
	that is induced by $\varphi$.
	Consider the $2$-cocycle
	$$
	\rho = \left(xz(xw)^{-1}, yz(yw)^{-1}, z^2(zw)^{-1}\right) \,\in\, I_{xw} \oplus I_{yw} \oplus I_{zw} \subset C_\mm^2 \otimes_R I
	$$
	of $C_\mm^\bullet \otimes_R I$ and notice that $\HL^2(\varphi)$ makes the mapping
	$$
	[\rho] \in \HL^2(I) \,\mapsto\, x^2(yw)^{-1} \in  \HL^2(\overline{R}).
	$$
	The element $x^2(yw)^{-1} \in  \HL^2(\overline{R})$ is non-zero.
	This implies that $\HL^2(\varphi) \neq 0$ and completes the proof. 
	We can verify that all the other tangent vectors in \cite[Lemma 6]{TWISTED_CUBIC} induce the zero map in cohomology.
\end{example}


Our final example involves the Hilbert scheme of curves of degree $4$ and genus $0$. It is shown in  \cite[Theorem 4.10]{DETACHING}
that every $[X] \in \Hilb^{4m+1}_{\PP^3_\kk}$ induces an exact sequence
$$
0 \to  \II_X/\II_{X'} \to \OO_X \to   \OO_{X'} \to 0
$$
where $X'$ is a locally Cohen-Macaulay curve of degree $4$ and $X'$ has arithmetic genus $0$, $1$ or $3$. If the genus of $X'$ is $0$ then $X$ is either an extremal or a subextremal curve \cite{NOLLET_SUB}. By  \cite[Proposition 2.1, Corollary 5.6]{NAGEL} the curve $X$ is an extremal curve if and only if
$$
h^1(\II_X(v)) =
	\begin{cases}
	0 & \text{ if } v \leq -1 \\
	1 & \text{ if } 0 \leq v \leq 2 \\
	0 & \text{ if } v \geq 3
	\end{cases}
	\quad \text{ and }	\quad
h^2(\II_X(v)) =
	\begin{cases}
	-1 - 4\nu  & \text{ if } v \leq -1 \\
	1 & \text{ if } v = 0 \\
	0 & \text{ if } v \geq 1.
	\end{cases}
$$
Using the ideal sheaf long exact sequence we deduce that $h^1(\OO_X(\nu)) = h^2(\II_X(\nu))$ and thus $h^0(\OO_X(\nu))=(4\nu+1)+h^1(\OO_X(\nu))$. Define the tuple of functions 
$\mathbf{h} = (h_1,h_2):\ZZ^2 \to \NN^2$ where
$$
	h_0(\nu) = 
	\begin{cases}
		0 & \text{ if } v \leq -1 \\
	2 &  \text{ if } v = 0 \\
		4v+1  & \text{ if } v \geq 1, \\
	\end{cases}
	\quad \text{ and }	\quad
	h_1(\nu) = 
	\begin{cases}
		-1-4v & \text{ if } v \leq -1 \\
	1 &  \text{ if } v = 0 \\
		0  & \text{ if } v \geq 1. \\
	\end{cases}
$$

\smallskip
We will now show that although the locus of extremal curves is singular in the Hilbert scheme \cite[Example 3.10]{NOLLET_THREE}, the corresponding fiber-full scheme, namely $\Fib^{(\mathbf{h},0,0)}_{\PP^3_\kk}$, is smooth.

\begin{example} Let $X$ be a subscheme parameterized by $\Fib^{(\mathbf{h},0,0)}_{\PP^3_\kk}$. By \cite[Proposition 0.6]{DESCHAMPS_PERRIN_2}, $X$ is a disjoint union of a line and a plane cubic or 
\begin{equation} \label{WUT}
I_{X} = (x_0^2,x_0x_1,gx_1^2,Gx_0+gFx_1)
\end{equation}
where $F,G$ are forms in $x_2,x_3$ of degrees $1$ and $3$ respectively, with no common zeros, while $g:=g(x_1,x_2,x_3)$ is a quadratic form in $x_1,x_2,x_3$. By \cite[Theorem 4.3]{DESCHAMPS_PERRIN_2} the fiber-full scheme $\Fib^{(\mathbf{h},0,0)}_{\PP^3_\kk}$ is a 16 dimensional, irreducible subset of $\Hilb^{4m+1}_{\PP^3_\kk}$. 

It is easy to see that a disjoint union of a line and a plane cubic is smooth inside the Hilbert scheme and thus also in the fiber-full scheme. We may assume that $I_X$ is expressed in the form of Equation \autoref{WUT}. Up to a change of coordinates we may assume $F= x_2$. 
Since $G$ and $x_2$ do not share a common zero we may write $G := G(x_2,x_3) = (\alpha_1x_2-x_3)(\alpha_2 x_2 - x_3)(\alpha_3 x_2 - x_3)$ with $\alpha_i \in \kk$. 
Consider the following family of projectively equivalent ideals
$$
\{(x_0^2,x_0x_1,g(x_1,tx_2,x_3)x_1^2,x_0G(tx_2,x_3) + g(x_1,tx_2,x_3)x_1x_2)\}_{t\in \bbA^1}.
$$
The flat limit of the above family is $(x_0^2,x_0x_1,g(x_1,0,x_3)x_1^2,x_0x^3_3 + g(x_1,0,x_3)x_1x_2)$ (this has the correct Hilbert polynomial since it is in the same format as Equation \autoref{WUT}). Similarly, we obtain 
$$
\lim_{t\to 0} (x_0^2,x_0x_1,g(x_1,0,tx_3)x_1^2,x_0x^3_3 + g(x_1,0,tx_3)x_1x_2 = I := (x_0^2,x_0x_1,x_1^4,x_0x_3^3 + x_1^3x_2).
$$
Thus, by upper-semicontinuity, it suffices to show that $\Fib^{(\mathbf{h},0,0)}_{\PP^3_\kk}$ is smooth at $[I]$.

We begin by computing the tangent space to $[I]$ on $\Hilb^{4m+1}_{\PP^3_\kk}$. Consider the truncation
$$
J := I(3) = \mm^3 \cap I = (x_0x_1x_3,x_0^2x_3,x_0x_1x_2,x_0^2x_2,x_0x_1^2,	x_0^2x_1,x_0^3,x_1^3x_2+x_0x_3^3,x_1^4).
$$ 
As done in \cite[Lemma 3.3, Proposition 3.5]{TWO_BOREL}, it is straightforward to check that $T_{[I]} \Hilb^{4m+1}_{\PP^3_\kk} = \left[\Hom(J,S/J)\right]_0$ and that a general $\varphi \in \left[\Hom(J,S/J)\right]_0$ is  given by
\begin{align*}
\varphi(x_0x_1x_3) & = -\gamma_{1}x_0x_3^2 + \gamma_2 x_1^3 - \gamma_3 x_0x_2x_3 
	+\delta_1 x_1^2x_3 +  \delta_2x_1x_2x_3 + \delta_3 x_1x_3^2 - \delta_4 x_1x_2x_3   \\
\varphi(x_0^2x_3) & =  
	2\delta_2x_0x_2x_3 + 2\delta_3 x_0x_3^2 - \delta_4 x_0x_2x_3   \\
\varphi(x_0x_1x_2) &= 
	- \gamma_1x_0x_2x_3 - \gamma_2 x_0x_3^2 - \gamma_3 x_0x_2^2
	+ \delta_1 x_1^2x_2 + \delta_2x_1x_2^2 + \delta_3 x_1x_2x_3    - \delta_4 x_1x_2^2  \\
\varphi(x_0^2x_2) & =  
	2\delta_2x_0x_2^2 + 2\delta_3 x_0x_2x_3 - \delta_4 x_0x_2^2  \\
\varphi(x_0x_1^2) & =   \delta_1 x_1^3 +  \delta_2x_1^2x_2 + \delta_3 x_1^2x_3 - \delta_4 x_1^2x_2 \\
\varphi(x_0^2x_1) & = 0 \\
\varphi(x_0^3) &=  0  \\
\varphi(x_1^3x_2+x_0x_3^3) &=  
	\alpha_1x_0x_2^3 + \alpha_2x_0x_2^2x_3 + \alpha_3x_0x_2x_3^2 + \alpha_4x_0x_3^3 + \alpha_5x_1^3x_3 \\
	& \quad 	
	+ \gamma_1x_1^2x_2x_3 + \gamma_2 x_1^2x_3^2 + \gamma_3 x_1^2x_2^2 \\
	& \quad
	+ \beta_1 x_1^2x_2^2 + \beta_2 x_1^2x_2x_3 +  \beta_3 x_1x_2^3  + \beta_4 x_1x_2^2x_3 + \beta_5 x_1x_2x_3^2
	 +\delta_1 x_1x_3^3 + \delta_2x_2x_3^3 + \delta_3 x_3^4  \\
\varphi(x_1^4) &=  -\beta_1 x_0x_3^3 +\beta_2 x_1^3x_3+\beta_3x_1^2x_2^2 +\beta_4x_1^2x_2x_3 +\beta_5 x_1^2x_3^2 + 
	\delta_4 x_1x_3^3 
\end{align*}
where $\alpha_1,\dots,\alpha_5,\beta_1,\dots \beta_5,\delta_1,\dots,\delta_5,\gamma_1,\gamma_2,\gamma_3 $ are independent parameters. In particular, $T_{[I]}\Hilb^{4m+1}_{\PP^3_\kk}$ is $17$-dimensional. Thus, to show that $[J]$ is smooth on its fiber-full scheme, it suffices to show that $\HL^2 (\varphi) : \HL^2(J) \rightarrow \HL^2(R/J)$ is non-zero for some $\varphi$. 

Let $\varphi$ be the map obtained by setting $\gamma_2 = 1$ and all other parameters to $0$.  Let $\overline{R} = R/J$, and note that the associated \v{C}ech complex  $C_\mm^\bullet \otimes_R \overline{R}$ collapses to
$$
0 \to \overline{R} \xrightarrow{{\begin{pmatrix} 1  & 1 \end{pmatrix}}^T}
	\overline{R}_{x_2} \oplus \overline{R}_{x_3} \xrightarrow{\begin{pmatrix} -1 & 1 \end{pmatrix}} \overline{R}_{x_2x_3}\to 0.
$$
Proceeding as in \autoref{TWISTED_CUBIC_EXAMPLE}, we can consider the $2$-cocycle
\begin{align*}
	\rho = \left(\frac{-x_0^3}{x_0^2x_2},0, \frac{-x_0x_1^2}{x_1^2x_2}, 0, 
			\frac{x_0x_3^3}{x_2x_3^3}\right)  + 
			\left(0, \frac{x_0x_1^3x_3}{x_0x_3^3}, 0, \frac{x_1^4}{x_1x_3^3}, 
			\frac{x_1^3x_2}{x_2x_3^3} \right)
			\\
			\,\in\, J_{x_0x_2} \oplus J_{x_0x_3} \oplus J_{x_1x_2} \oplus J_{x_1x_3} \oplus J_{x_2x_3} \subset C_\mm^2 \otimes_R J
\end{align*}
	of $C_\mm^\bullet \otimes_R J$ and notice that $\HL^2(\varphi)$ makes the mapping
	$$
	[\rho] \in \HL^2(J) \,\mapsto\, \frac{x_1^2x_3^2}{x_2x_3^3} = \frac{x_1^2}{x_2x_3}  \in  \HL^2(\overline{R}).
	$$
It is easy to see that the above map is non-zero. Thus $[J]$ is a smooth point on the fiber-full scheme, concluding this example.
\end{example}

\section*{Acknowledgments}

We thank Joachim Jelisiejew for helpful comments and correspondence.
Y.C.R. was partially supported by an FWO Postdoctoral Fellowship (1220122N).
We thank the reviewer for helpful comments and suggestions.


\end{document}